\documentclass[10pt,a4paper]{article}
\usepackage[margin=1.2in]{geometry}
\usepackage{amssymb,latexsym,amsmath,amsfonts,amsthm,scalerel,mathtools}
\usepackage{graphicx,comment}
\usepackage{epsfig}
\usepackage{tikz}
\usepackage{etoolbox}
\usepackage{comment}
\usepackage[T1]{fontenc}
\def\BigColSep{\setlength{\arraycolsep}{1pt}}

\newcommand{\supp}{\mathrm{supp}}

\DeclareMathOperator*{\Kont}{\scalerel*{\Huge\mathrm{K}}{\big(}}

\newtheorem{theorem}{Theorem}[section]
\newtheorem{lemma}[theorem]{Lemma}
\newtheorem{proposition}[theorem]{Proposition}

\newtheorem{corollary}[theorem]{Corollary}

\theoremstyle{definition}

\theoremstyle{remark}

\newtheorem{remark}[theorem]{Remark}

\numberwithin{equation}{section}

\title{Lattice paths, vector continued fractions, and resolvents of banded Hessenberg operators}

\author{A. L\'{o}pez-Garc\'{i}a \qquad V.A. Prokhorov}

\date{\today}

\begin{document}

\maketitle

\begin{abstract}
We give a combinatorial interpretation of vector continued fractions obtained by applying the Jacobi--Perron algorithm to a vector of $p\geq 1$ resolvent functions of a banded Hessenberg operator of order $p+1$. The interpretation consists in the identification of the coefficients in the power series expansion of the resolvent functions as weight polynomials associated with Lukasiewicz lattice paths in the upper half-plane. In the scalar case $p=1$ this reduces to the relation established by P. Flajolet and G. Viennot between Jacobi--Stieltjes continued fractions, their power series expansion, and Motzkin paths. We consider three classes of lattice paths, namely the Lukasiewicz paths in the upper half-plane, their symmetric images in the lower half-plane, and a third class of unrestricted lattice paths which are allowed to cross the $x$-axis. We establish a relation between the three families of paths by means of a relation between the associated generating power series. We also discuss the subcollection of Lukasiewicz paths formed by the partial $p$-Dyck paths, whose weight polynomials are known in the literature as genetic sums or generalized Stieltjes--Rogers polynomials, and express certain moments of bi-diagonal Hessenberg operators.   

\smallskip


\smallskip

\end{abstract}

\section{Introduction}

\subsection{A background in the scalar case}

In this work we present a combinatorial interpretation for a class of vector continued fractions. We do this by establishing a relation between the collection of Lukasiewicz lattice paths in the upper half-plane and vector continued fraction expansions for vectors of resolvent functions of banded Hessenberg operators. Prior to detailing this connection and the main results of our work, we review some fundamental ideas in combinatorics and approximation theory that motivated our investigation. 

Let $\mu$ be a probability measure on the real line with infinite and compact support, and consider the sequence $(P_{n})_{n=0}^{\infty}$ of monic orthogonal polynomials associated with $\mu$. It is well known that these polynomials satisfy a three-term recurrence relation
\begin{equation}\label{ttrl}
xP_{n}(x)=P_{n+1}(x)+b_{n} P_{n}(x)+a_{n-1} P_{n-1}(x),\qquad n\geq 1,
\end{equation}
with initial values $P_{0}(x)=1$, $P_{1}(x)=x-b_{0}$. The coefficients $\{a_{n},b_{n}\}_{n\geq 0}$ in \eqref{ttrl} are the Jacobi parameters for $\mu$. The recurrence relation \eqref{ttrl} expressed in matrix form is
\[
x\begin{pmatrix}
P_{0}(x) \\
P_{1}(x) \\
P_{2}(x) \\
\vdots
\end{pmatrix}=\begin{pmatrix}
b_{0} & 1 & & & \\
a_{0} & b_{1} & 1 & \\
 & a_{1} & b_{2} & 1 &  \\
 &  &  \ddots & \ddots & \ddots    
\end{pmatrix}
\begin{pmatrix}
P_{0}(x) \\
P_{1}(x) \\
P_{2}(x) \\
\vdots
\end{pmatrix}
\]
and the tridiagonal matrix on the right-hand side is the infinite Jacobi matrix associated with $\mu$, which we denote $J=J(\mu)$. The spectral theorem for such operators (see e.g. \cite{NikSor}) establishes the relation
\begin{equation}\label{spectraltheo}
\int\frac{1}{z-x}\,d\mu(x)=\langle(zI-J)^{-1}e_{0},e_{0}\rangle,\qquad z\in\mathbb{C}\setminus\supp(\mu),
\end{equation}
where $(zI-J)^{-1}$ is the resolvent of $J$ and $\{e_{j}\}_{j=0}^{\infty} $ is the standard orthonormal basis in $\ell^{2}(\mathbb{Z}_{\geq 0})$. The identity \eqref{spectraltheo} implies 
\begin{equation}\label{eq:moments}
s_{n}:=\int x^{n}d\mu(x)=\langle J^{n} e_{0},e_{0}\rangle,\qquad n\geq 0,
\end{equation}
that is, the $(0,0)$-entry of the matrix $J^{n}$ coincides with the $n$-th moment of $\mu$. The relation \eqref{eq:moments} shows that the moments of $\mu$ are determined by the Jacobi parameters and can be expressed as polynomials in these parameters with positive integer coefficients. The first few expressions are
\begin{align*}
s_{1} & =b_{0}\\
s_{2} & =b_{0}^2+a_{0}\\
s_{3} & =b_{0}^3+2 a_{0} b_{0}+a_{0} b_{1}\\
s_{4} & =b_{0}^4+3 a_{0} b_{0}^2+2 a_{0} b_{0} b_{1}+a_{0} b_{1}^{2}+a_{0}^2+a_{0}a_{1}.
\end{align*}
An important result of Viennot \cite{Viennot} was a combinatorial interpretation of the moments $s_{n}$ in terms of Motzkin paths, which we now describe.

A Motzkin path is a lattice path with vertices $(n,m)$ in the lattice $\mathbb{Z}_{\geq 0}\times \mathbb{Z}_{\geq 0}$ that starts at $(0,0)$, ends on the $x$-axis, and has edges $v\rightarrow v'$ between adjacent vertices in the path of only three types:
\begin{align*}
& \mbox{upsteps}\,\,\,(n,m)\rightarrow(n+1,m+1),\\
& \mbox{level steps}\,\,\,(n,m)\rightarrow(n+1,m),\\
& \mbox{downsteps}\,\,\,(n,m)\rightarrow(n+1,m-1),\,\,\,\,m\geq 1.
\end{align*}
See an example in Fig.~\ref{Motzkinpath}. Let $\mathcal{M}_{n}$ denote the set of all Motzkin paths from $(0,0)$ to $(n,0)$. We give to each path $\gamma\in\mathcal{M}_{n}$ a weight $w(\gamma)$ defined as the product of the weights of the $n$ individual edges in the path, and the weight of an edge is defined by
\begin{align*}
w((n,m)\rightarrow (n+1,m+1)) & =1,\\
w((n,m)\rightarrow (n+1,m)) & = b_{m},\\
w((n,m)\rightarrow (n+1,m-1)) & =a_{m-1},\,\,\,\,m\geq 1.
\end{align*} 
Viennot \cite{Viennot} proved that for each $n\geq 1$ we have
\begin{equation}\label{Vieniden}
s_{n}=\sum_{\gamma\in\mathcal{M}_{n}}w(\gamma).
\end{equation}
The expression $\sum_{\gamma\in\mathcal{M}_{n}}w(\gamma)$ is therefore a polynomial in the Jacobi parameters, which we call weight polynomial associated with $\mathcal{M}_{n}$ (for more details see subsection~\ref{subsecvsmr}). The orthogonal polynomials $P_{n}(x)$ themselves can also be expressed by an identity of the form \eqref{Vieniden} (using a different class of lattice paths), see e.g. Stanton \cite{Stanton}.  

A central result in the theory of orthogonal polynomials and Pad\'{e} approximation is the fact that the moment generating function 
\[
\int\frac{1}{z-x}d\mu(x)=\sum_{n=0}^{\infty}\frac{s_{n}}{z^{n+1}}
\]
has the Jacobi--Stieltjes continued fraction expansion
\begin{equation}\label{Chebcontfrac}
\int\frac{1}{z-x}\,d\mu(x)=\cfrac{1}{z-b_{0}-\cfrac{a_0}{z-b_{1}-\cfrac{a_{1}}{\ddots}}}
\end{equation}
which in virtue of Markov's theorem \cite[pg. 89]{Markov} converges at every point $z$ in the complex plane outside the convex hull of $\supp(\mu)$. Remarkably, the convergents of the continued fraction are the rational functions
\[
\pi_{n}(z)=\frac{Q_{n}(z)}{P_{n}(z)}
\]
were the denominators are the orthogonal polynomials associated with $\mu$, and the numerators are the polynomials of the second kind
\[
Q_{n}(z)=\int\frac{P_{n}(z)-P_{n}(x)}{z-x}\,d\mu(x).
\] 
Observe that these polynomials are also solutions of \eqref{ttrl} with the initial values $Q_{0}(x)=0$, $Q_{1}(x)=1$. The convergents $\pi_{n}(z)$ are also the diagonal Pad\'{e} approximants of $\int(z-x)^{-1}d\mu(x)$.

\begin{figure}
\begin{center}
\begin{tikzpicture}[scale=0.7]
\draw[line width=1.5pt]  (-3.03,0) -- (10.5,0);
\draw[line width=1.5pt]  (-3,0) -- (-3,3.5);
\draw[line width=0.7pt] (-3,0) -- (-2,1) -- (-1,1) -- (0,0) -- (1,1) -- (2,1) -- (3,2) -- (4,1) -- (5,1) -- (6,2) -- (7,2) -- (8,1) -- (9,1) -- (10,0);
\draw [line width=0.5] (-2,0) -- (-2,-0.12);
\draw [line width=0.5] (-1,0) -- (-1,-0.12);
\draw [line width=0.5] (0,0) -- (0,-0.12);
\draw [line width=0.5] (1,0) -- (1,-0.12);
\draw [line width=0.5] (2,0) -- (2,-0.12);
\draw [line width=0.5] (3,0) -- (3,-0.12);
\draw [line width=0.5] (4,0) -- (4,-0.12);
\draw [line width=0.5] (5,0) -- (5,-0.12);
\draw [line width=0.5] (6,0) -- (6,-0.12);
\draw [line width=0.5] (7,0) -- (7,-0.12);
\draw [line width=0.5] (8,0) -- (8,-0.12);
\draw [line width=0.5] (9,0) -- (9,-0.12);
\draw [line width=0.5] (10,0) -- (10,-0.12);
\draw [line width=0.5] (-3,1) -- (-3.12,1);
\draw [line width=0.5] (-3,2) -- (-3.12,2);
\draw [line width=0.5] (-3,3) -- (-3.12,3);
\draw [dotted] (-3,1) -- (10.5,1);
\draw [dotted] (-3,2) -- (10.5,2);
\draw [dotted] (-3,3) -- (10.5,3);
\draw [dotted] (-2,0) -- (-2,3.5);
\draw [dotted] (-1,0) -- (-1,3.5);
\draw [dotted] (0,0) -- (0,3.5);
\draw [dotted] (1,0) -- (1,3.5);
\draw [dotted] (2,0) -- (2,3.5);
\draw [dotted] (3,0) -- (3,3.5);
\draw [dotted] (4,0) -- (4,3.5);
\draw [dotted] (5,0) -- (5,3.5);
\draw [dotted] (6,0) -- (6,3.5);
\draw [dotted] (7,0) -- (7,3.5);
\draw [dotted] (8,0) -- (8,3.5);
\draw [dotted] (9,0) -- (9,3.5);
\draw [dotted] (10,0) -- (10,3.5);
\draw (-2,-0.1) node[below, scale=0.8]{$1$};
\draw (-1,-0.1) node[below, scale=0.8]{$2$};
\draw (0,-0.1) node[below, scale=0.8]{$3$};
\draw (1,-0.1) node[below, scale=0.8]{$4$};
\draw (2,-0.1) node[below, scale=0.8]{$5$};
\draw (3,-0.1) node[below, scale=0.8]{$6$};
\draw (4,-0.1) node[below, scale=0.8]{$7$};
\draw (5,-0.1) node[below, scale=0.8]{$8$};
\draw (6,-0.1) node[below, scale=0.8]{$9$};
\draw (7,-0.1) node[below, scale=0.8]{$10$};
\draw (8,-0.1) node[below, scale=0.8]{$11$};
\draw (9,-0.1) node[below, scale=0.8]{$12$};
\draw (10,-0.1) node[below, scale=0.8]{$13$};
\draw (-3.1,1) node[left, scale=0.8]{$1$};
\draw (-3.1,2) node[left, scale=0.8]{$2$};
\draw (-3.1,3) node[left, scale=0.8]{$3$};
\end{tikzpicture}
\end{center}
\caption{Example of a Motzkin path. The weight of this path is $a_{0}^2\,a_{1}^2\, b_{1}^4\, b_{2}$.}
\label{Motzkinpath}
\end{figure}
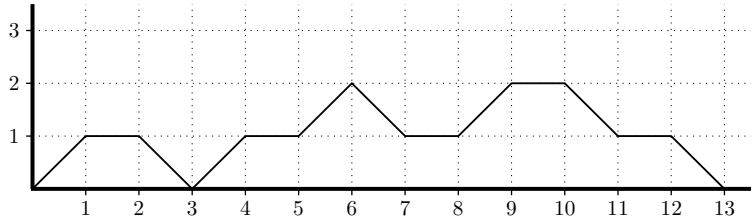

Let us now discuss the relation established by P. Flajolet between Jacobi--Stieltjes continued fractions, their power series expansion, and Motzkin paths. Consider two arbitrary sequences of complex numbers $(a_{n})_{n=0}^{\infty}$ and $(b_{n})_{n=0}^{\infty}$, and let $J$ be the associated infinite Jacobi matrix  
\[
J=\begin{pmatrix}
b_{0} & 1 & & & \\
a_{0} & b_{1} & 1 & \\
 & a_{1} & b_{2} & 1 &  \\
 &  &  \ddots & \ddots & \ddots    
\end{pmatrix}.
\]
Denote by
\[
m(z)=\sum_{n=0}^{\infty}\frac{A_{n}}{z^{n+1}}
\]
the formal power series generated by the sequence of weight polynomials
\[
A_{n}=\sum_{\gamma\in\mathcal{M}_{n}}w(\gamma).
\]
Let $\mathcal{M}_{n}^{(1)}:=\{\gamma+1: \gamma\in\mathcal{M}_{n}\}$, where $\gamma+1$ denotes the path obtained by shifting $\gamma$ vertically $1$ unit upwards. So the paths in $\mathcal{M}_{n}^{(1)}$ start at the point $(0,1)$, they end on the line $y=1$, and have no vertex below that line. Let $A_{n}^{(1)}$ be the weight polynomial associated with $\mathcal{M}_{n}^{(1)}$
\[
A_{n}^{(1)}=\sum_{\gamma\in\mathcal{M}_{n}^{(1)}}w(\gamma),
\]
and let
\[
m_{1}(z)=\sum_{n=0}^{\infty}\frac{A_{n}^{(1)}}{z^{n+1}}.
\]
Flajolet \cite{Fla} gave an explicit formula for the weight polynomials $A_{n}$ (which he called \textit{Jacobi-Rogers polynomials}) in terms of the Jacobi parameters and binomial coefficients (see Proposition 3.A in \cite{Fla}), and gave a combinatorial proof of the following relation
\begin{equation}\label{relmfunc}
m(z)=\frac{1}{z-b_{0}-a_{0}\,m_{1}(z)}.
\end{equation}
As a consequence of \eqref{relmfunc} we obtain the formal identity
\[
m(z)=\cfrac{1}{z-b_{0}-\cfrac{a_0}{z-b_{1}-\cfrac{a_{1}}{\ddots}}}.
\]
We present the proof of \eqref{relmfunc} now since it is a very good synthesis of the ideas (partitioning of collections of paths, decomposition of paths, effect of horizontal and vertical translations of paths in the weights) we will use in the more general vector setting.

First, observe that \eqref{relmfunc} is equivalent to $z m(z)-1=(b_{0}+a_{0}\,m_{1}(z))\,m(z)$. Identifying  coefficients in the series on both sides we see that this is in turn equivalent to 
\begin{equation}\label{relmoments}
A_{n}=b_{0} A_{n-1}+a_{0}\sum_{k=0}^{n-2} A_{k}^{(1)} A_{n-k-2},\qquad n\geq 1.
\end{equation}
The proof of the relation \eqref{relmoments} goes as follows. Subdivide $\mathcal{M}_{n}$ into two disjoint subsets $\mathcal{A}_{n}$ and $\mathcal{B}_{n}$, where $\mathcal{A}_{n}$ consists of those paths with first step the segment $(0,0)\rightarrow(1,0)$, and $\mathcal{B}_{n}$ consists of those paths with first step the segment $(0,0)\rightarrow(1,1)$. So we have
\[
A_{n}=\sum_{\gamma\in\mathcal{M}_{n}}w(\gamma)=\sum_{\gamma\in\mathcal{A}_{n}}w(\gamma)+\sum_{\gamma\in\mathcal{B}_{n}}w(\gamma).
\]
Since $w((0,0)\rightarrow(1,0))=b_{0}$, it is clear that $\sum_{\gamma\in\mathcal{A}_{n}}w(\gamma)=b_{0} A_{n-1}$. To see that $\sum_{\gamma\in\mathcal{B}_{n}} w(\gamma)=a_{0}\sum_{k=0}^{n-2}A_{k}^{(1)}A_{n-k-2}$, take an arbitrary path $\gamma\in\mathcal{B}_{n}$. Since $\gamma$ intersects the line $y=1$ and it ends on the $x$-axis, there must be at least one downstep in $\gamma$ with terminal point on the $x$-axis. Suppose that $(k+1,1)\rightarrow(k+2,0)$ is the \textit{first} downstep in $\gamma$ (from left to right) with such property. It is obvious that $0\leq k\leq n-2$. We subdivide $\gamma$ into four parts $\gamma_{j}$, $1\leq j\leq 4$, where $\gamma_{1}$ is the initial step $(0,0)\rightarrow(1,1)$ with weight $1$, $\gamma_{2}$ is the portion of $\gamma$ on the interval $1\leq x\leq k+1$ (which can be identified with a horizontal translation of a path in $\mathcal{M}_{k}^{(1)}$), $\gamma_{3}$ is the special downstep $(k+1,1)\rightarrow(k+2,0)$ with weight $a_{0}$, and $\gamma_{4}$ is the portion of $\gamma$ on the interval $k+2\leq x\leq n$, which is a horizontal translation of a path in $\mathcal{M}_{n-k-2}$. If we do this subdivision for each $\gamma\in\mathcal{B}_{n}$, and take into account that the weight of a path is invariant under horizontal translations, we get 
\[
\sum_{\gamma\in\mathcal{B}_{n}}w(\gamma)=\sum_{\gamma\in\mathcal{B}_{n}}w(\gamma_{1})w(\gamma_{2})w(\gamma_{3})w(\gamma_{4})=a_{0}\sum_{\gamma\in\mathcal{B}_{n}}w(\gamma_2)w(\gamma_4)=a_{0}\sum_{k=0}^{n-2} A_{k}^{(1)}A_{n-k-2},
\]  
and the proof of \eqref{relmoments} is finished.

It follows from our discussion that for each $n\geq 0$ we have
\begin{align*}
A_{n} & =\langle J^n e_{0},e_{0}\rangle,\\
A_{n}^{(1)} & =\langle J_{1}^n e_{0},e_{0}\rangle,
\end{align*}
where $J_{1}$ denotes the infinite Jacobi matrix obtained by removing the first row and the first column of $J$.

\subsection{The vector setting and our main results}\label{subsecvsmr}

In the late 1970s, the approximation theory school led by A. A. Gonchar and E. M. Nikishin began a systematic study of the asymptotic properties of Hermite-Pad\'{e} approximants to systems of analytic functions. One of the first works in this area was the influential paper of Nikishin \cite{Nik}, in which he obtained an analogue of Markov's theorem for Hermite-Pad\'{e} approximants to systems of functions $(\phi_{0}(z),\ldots,\phi_{p-1}(z))$ of the form
\begin{equation}\label{Markovfunc}
\phi_{j}(z)=\int\frac{d\mu_{j}(x)}{z-x},\qquad 0\leq j\leq p-1,
\end{equation}
where $(\mu_{0},\ldots,\mu_{p-1})$ is a vector of $p\geq 2$ positive measures on the real line that form a so-called Nikishin system. 

In many problems of interest the functions \eqref{Markovfunc} can be identified as resolvent functions of a banded Hessenberg operator. Kalyagin \cite{Kal,Kal2} investigated the problem of Hermite-Pad\'{e} approximation to such functions and obtained a vector continued fraction representation for the approximants and the vector of resolvent functions. We describe now his result, which is directly related to our work.

In this paper, $\mathbb{C}((z^{-1}))$ will denote the set of all formal power series
\[
\sum_{n\in\mathbb{Z}}\frac{a_{n}}{z^{n}}
\]
with complex coefficients such that only finitely many $a_{n}$ with $n<0$ are non-zero. This set is an algebraic field with the usual addition and product of series. 

Let $p\geq 1$ be an integer, which throughout the paper will remain fixed. For each $0\leq k\leq p$, let $(a_{n}^{(k)})_{n\in\mathbb{Z}}$ be a bi-infinite sequence of complex numbers. With the positive part of these sequences (later the whole sequence will be used) we construct the infinite matrix 
\begin{equation}\label{Hmatrix}
H=(h_{i,j})_{i,j=0}^{\infty}=\begin{pmatrix}
a_{0}^{(0)} & 1 & & & \\
\vdots & a_{1}^{(0)} & 1 & \\
a_{0}^{(p)} & \vdots & a_{2}^{(0)} & \ddots \\
 & a_{1}^{(p)} & \vdots & \ddots \\
 &  & a_{2}^{(p)} & \\
 & & & \ddots \\
 &   
\end{pmatrix}
\end{equation}
with entries
\[
\begin{cases}
h_{j-1,j}=1, & j\geq 1,\\
h_{j+k,j}=a_{j}^{(k)}, & 0\leq k\leq p, \quad j\geq 0,\\
h_{i,j}=0, & \mbox{otherwise}.
\end{cases}
\]
Note that this is a banded lower Hessenberg matrix. Recall that $\{e_{j}\}_{j=0}^{\infty}$ denotes the standard basis in $\ell^{2}(\mathbb{Z}_{\geq 0})$, and consider the resolvent functions
\begin{equation}\label{resolvfuncphi}
\phi_{j}(z):=\langle(zI-H)^{-1}e_{j},e_{0}\rangle=\sum_{n=0}^{\infty}\frac{\langle H^{n}e_{j},e_{0}\rangle}{z^{n+1}},\qquad 0\leq j\leq p-1,
\end{equation}
understood as formal power series. 

Hermite-Pad\'{e} approximants to the functions \eqref{resolvfuncphi} can be constructed as follows. Let $H_{n}=(h_{i,j})_{i,j=0}^{n-1}$ denote the principal $n\times n$ truncation of $H$, and let $H_{n}^{[k]}$, $1\leq k\leq p$, be the submatrix of $H_{n}$ obtained by removing the first $k$ rows and columns of $H_{n}$. For each $n\geq 0$ we define the characteristic polynomials
\begin{equation}\label{charpoly}
\begin{aligned}
q_{n}(z) & :=\det(z I_{n}-H_{n}),\\
q_{n,k}(z) & :=\det(z I_{n-k}-H_{n}^{[k]}),\qquad 1\leq k\leq p.
\end{aligned}
\end{equation} 
Kalyagin \cite{Kal,Kal2} proved that the vector of rational functions
\begin{equation}\label{vectHPapprox}
\left(\frac{q_{n,1}(z)}{q_{n}(z)},\frac{q_{n,2}(z)}{q_{n}(z)},\ldots,\frac{q_{n,p}(z)}{q_{n}(z)}\right)
\end{equation}
is an Hermite-Pad\'{e} approximant at infinity of order $n$ to the system $(\phi_{0}(z),\ldots,\phi_{p-1}(z))$, that is
\begin{equation}\label{HPproperty}
q_{n}(z)\,\phi_{k}(z)-q_{n,k+1}(z)=O(z^{-n_{k}-1}),\quad z\rightarrow\infty,\quad 0\leq k\leq p-1,
\end{equation}
where $n_{k}=\lfloor (n-k)/p\rfloor$. If the resolvent functions $\phi_{k}(z)$ have the form \eqref{Markovfunc}, then it easily follows from \eqref{HPproperty} that the polynomials $q_{n}$ are multiple orthogonal polynomials with respect to $(\mu_{0},\ldots,\mu_{p-1})$, see \cite{NikSor}. It is also easy to see that the $p+1$ sequences in \eqref{charpoly} form a basis for the solution space of the difference equation 
\[
z y_{n}=y_{n+1}+a_{n}^{(0)}\,y_{n}+a_{n-1}^{(1)}\,y_{n-1}+\cdots+a_{n-p}^{(p)}\, y_{n-p},\qquad n\geq p.
\]
If $p=1$ and $H=J$ is a Jacobi matrix, then $q_{n,1}/q_{n}$ is the $n$-th diagonal Pad\'{e} approximant to the function $\langle(zI-J)^{-1}e_{0},e_{0}\rangle$.

The construction of a vector continued fraction is based on the following division operation in $\mathbf{F}^{p}$, $\mathbf{F}=\mathbb{C}((z^{-1}))$. If $(f_{1},\ldots,f_{p})$ and $(g_{1},\ldots,g_{p})$ are two vectors of formal power series and $g_{p}\neq 0$, then we define
\begin{equation}\label{divoper}
\frac{(f_{1},\ldots,f_{p})}{(g_{1},\ldots,g_{p})}:=\left(\frac{f_{1}}{g_{p}},\frac{f_{2}\,g_{1}}{g_{p}},\frac{f_{3}\,g_{2}}{g_{p}},\ldots,\frac{f_{p}\,g_{p-1}}{g_{p}}\right).
\end{equation}
If $\mathbf{a}_{1},\ldots,\mathbf{a}_{n}$ and $\mathbf{b}_{1},\ldots,\mathbf{b}_{n}$ are now vectors in $\textbf{F}^{p}$, then we can form the finite continued fraction
\begin{equation}\label{fvcf}
\Kont_{m = 1}^{n}\Bigl(\frac{\mathbf{a}_{m}}{\mathbf{b}_{m}}\Bigr)\coloneqq\cfrac{\mathbf{a}_1}{\mathbf{b}_1+\cfrac{\mathbf{a}_2}{\mathbf{b}_2+\cfrac{\mathbf{a}_3}{\ddots \raisebox{-0.9em}{\ensuremath{+\cfrac{\mathbf{a}_n}{\mathbf{b}_n}}}}}}
\end{equation}
provided that each division can be performed. 

Before we describe Kalyagin's formulae, we need some additional definitions. Let $H^{(k)}$, $k\geq 0$, denote the infinite Hessenberg matrix obtained by deleting the first $k$ rows and columns of the matrix $H$ in \eqref{Hmatrix}, and analogously to \eqref{resolvfuncphi} we define
\[
\phi_{j}^{(k)}(z):=\langle(zI-H^{(k)})^{-1}e_{j},e_{0}\rangle=\sum_{n=0}^{\infty}\frac{\langle (H^{(k)})^{n}e_{j},e_{0}\rangle}{z^{n+1}},\qquad 0\leq j\leq p-1.
\]
Let
\begin{equation}\label{cks}
\mathbf{c}_{k}:=\begin{cases}
(1,1,\ldots,1), & 1\leq k\leq p,\\[0.5em]
(-a_{k-p-1}^{(p)},1,\ldots,1), & k\geq p+1,
\end{cases}
\end{equation}
\begin{equation}\label{dks}
\mathbf{d}_{k}(z):=\begin{cases}
(0,\ldots,0,-a_{0}^{(k-1)},-a_{1}^{(k-2)},\ldots,-a_{k-2}^{(1)},z-a_{k-1}^{(0)}), & 1\leq k\leq p,\\[0.5em]
(-a_{k-p}^{(p-1)},-a_{k-p+1}^{(p-2)},-a_{k-p+2}^{(p-3)},\ldots,-a_{k-2}^{(1)},z-a_{k-1}^{(0)}), & k\geq p+1,
\end{cases}
\end{equation}
\begin{equation}\label{vks}
\mathbf{v}_{k}(z):=\begin{cases}
(\phi_{0}^{(k)},\ldots,\phi_{p-k-1}^{(k)},-\sum_{j=k}^{p} a_{0}^{(j)}\phi_{j-k}^{(k)},\ldots,-\sum_{j=1}^{p} a_{k-1}^{(j)} \phi_{j-1}^{(k)}), & 0\leq k\leq p,\\[0.5em]
(-a_{k-p}^{(p)} \phi_{0}^{(k)},-\sum_{j=p-1}^{p} a_{k-p+1}^{(j)} \phi_{j-p+1}^{(k)},\ldots,-\sum_{j=1}^{p} a_{k-1}^{(j)} \phi_{j-1}^{(k)}), & k\geq p+1.
\end{cases}
\end{equation}

\begin{theorem}[Kalyagin \cite{Kal,Kal2}]
For every $n\geq 1$ we have
\begin{align}
(\phi_{0}(z),\ldots,\phi_{p-1}(z)) & =\Kont_{j = 1}^{n}\Bigl(\frac{\mathbf{c}_{j}}{\widetilde{\mathbf{d}}_{j}(z)}\Bigr),\label{vcfphifinite}\\
\left(\frac{q_{n,1}(z)}{q_{n}(z)},\frac{q_{n,2}(z)}{q_{n}(z)},\ldots,\frac{q_{n,p}(z)}{q_{n}(z)}\right) & =\Kont_{j = 1}^{n}\Bigl(\frac{\mathbf{c}_{j}}{\mathbf{d}_{j}(z)}\Bigr),\label{vcfHP}
\end{align}
where $\widetilde{\mathbf{d}}_{j}(z)=\mathbf{d}_{j}(z)$ if $j\leq n-1$ and $\widetilde{\mathbf{d}}_{n}(z)=\mathbf{d}_{n}(z)+\mathbf{v}_{n}(z)$.
\end{theorem} 

As a consequence of \eqref{vcfphifinite} and \eqref{vcfHP} we have the formal identity
\begin{equation}\label{vcfphi}
(\phi_{0}(z),\ldots,\phi_{p-1}(z))=\Kont_{j = 1}^{\infty}\Bigl(\frac{\mathbf{c}_{j}}{\mathbf{d}_{j}(z)}\Bigr).
\end{equation}
Observe that if $p=1$, then $\mathbf{c}_{1}=1$, $\mathbf{c}_{k}=-a_{k-2}^{(1)}$, $k\geq 2$, and $\mathbf{d}_{k}(z)=z-a_{k-1}^{(0)}$, $k\geq 1$, so \eqref{vcfphi} reduces to \eqref{Chebcontfrac}.

Vector continued fractions originated in the works of Jacobi \cite{Jac} and Perron \cite{Perron} in analytic number theory, where they investigated algorithms that generate simultaneous rational approximants to vectors of real constants (see also \cite{Bern} for an interesting account of that work). In function theory, the so-called Jacobi--Perron algorithm gives a continued fraction expansion for a vector of power series, where the coefficients in the continued fraction are vectors of polynomials. Such algorithm is used to obtain \eqref{vcfphi}. The literature on functional vector continued fractions is very extensive; some important works on their algebraic and convergence properties are \cite{AptKalVan,BCK,Kal,Kal2,Parus,VanIseg,VanIseg4}, see also \cite{NikSor} and references therein.

In this paper we give a combinatorial proof of the identity \eqref{vcfphifinite} based on a relation between the moments $\langle (H^{(k)})^n e_{j},e_{0}\rangle$ and the class of partial $p$-Lukasiewicz paths with vertices in the lattice $\mathbb{Z}_{\geq 0}\times\mathbb{Z}_{\geq 0}$. The proof is analogous to Flajolet's justification of \eqref{relmfunc} based on Viennot's identity \eqref{Vieniden}. The relation in question is the identification of the moments $\langle (H^{(k)})^n e_{j},e_{0}\rangle$ as \emph{weight polynomials} (polynomials in the parameters $a_{n}^{(k)}$) associated with certain collections of Lukasiewicz paths. The vector continued fraction is obtained as a consequence of some identities between the generating functions of these weight polynomials. The generating functions are defined as formal power series in $\mathbb{C}((z^{-1}))$ and the identities are described in Theorem~\ref{theo:relseriesAs}. 

In this paper we also consider two more classes of lattice paths; one on the lattice $\mathbb{Z}_{\geq 0}\times \mathbb{Z}$ (i.e., with paths allowed to go above and below the $x$-axis) and the other one on the lattice $\mathbb{Z}_{\geq 0}\times \mathbb{Z}_{\leq 0}$ (i.e., with paths in the lower half-plane). The third class is in bijection with the class of Lukasiewicz paths and can be obtained by a symmetry transformation of those paths. In Theorem~\ref{theo:relseriesWs} we obtain a relation between the generating functions associated with the three classes of lattice paths considered. This relation can also be interpreted as a relation between resolvent functions of two-sided and one-sided banded Hessenberg operators.   

We introduce now the definitions necessary to state our main results.

Recall that we have fixed bi-infinite sequences of complex numbers $(a^{(k)}_{n})_{n\in\mathbb{Z}}$, $0\leq k\leq p$. The lattice paths we construct have vertices on the set $\mathcal{V}:=\mathbb{Z}_{\geq 0}\times \mathbb{Z}$ and edges that belong to the set
\[
\mathcal{E}:=\mathcal{E}_{u}\cup\mathcal{E}_{\ell}\cup\mathcal{E}_{d},
\]
where
\begin{equation}\label{defsteps}
\begin{aligned}
\mathcal{E}_{u} & :=\{(n,m)\rightarrow (n+1,m+1) : (n,m)\in\mathcal{V}\}\\
\mathcal{E}_{\ell} & :=\{(n,m)\rightarrow (n+1,m) : (n,m)\in\mathcal{V}\}\\
\mathcal{E}_{d} & :=\{(n,m)\rightarrow (n+1,m-j) : (n,m)\in\mathcal{V},\,\,\,1\leq j\leq p\}.
\end{aligned}
\end{equation}
Here, as before, $v\rightarrow v'$ indicates the edge from vertex $v$ to vertex $v'$. We will refer to the edges in $\mathcal{E}_{u}$, $\mathcal{E}_{\ell}$, and $\mathcal{E}_{d}$ as the \emph{upsteps}, \emph{level steps}, and \emph{downsteps} respectively. Edges are generically called \emph{steps}. Observe that the difference in height between the vertices in a downstep is a value $j\in\{1,\ldots,p\}$. A lattice path is a finite sequence of steps
\begin{equation}\label{eq:path}
\gamma=e_{1}e_{2}\cdots e_{k},
\end{equation}
where for each $1\leq j\leq k-1,$ the final vertex of $e_{j}$ conicides with the initial vertex of $e_{j+1}$. We say that the path in \eqref{eq:path} has \emph{length} $k$. A path of length zero is simply a vertex in $\mathcal{V}$. If $(n,m)\in\mathcal{V}$ is a vertex in the path $\gamma$, we say that $\gamma$ has height $m$ at $n$. We define $\max(\gamma)$ to be the maximum of the heights of all vertices in $\gamma$, and $\min(\gamma)$ to be the minimum of the heights of all vertices in $\gamma$. Also, if $q\in\mathbb{Z}$ and $\gamma$ is a path, we denote by $\gamma+q$ the path obtained by shifting $\gamma$ upwards (or downwards) $|q|$ units if $q$ is positive (or negative). 

To each step we associate a \emph{weight} as follows:
\begin{equation}\label{weightedges}
\begin{aligned}
w((n,m)\rightarrow (n+1,m+1)) & =1,\\
w((n,m)\rightarrow (n+1,m-j)) & =a_{m-j}^{(j)},\qquad 0\leq j\leq p.
\end{aligned}
\end{equation}
Note that all upsteps have weight 1, regardless of their position. If $\gamma$ is a path, its weight is defined by
\begin{equation}\label{weightpath}
w(\gamma)=\prod_{e\subset\gamma}w(e),
\end{equation}
where the product runs over the different steps of $\gamma$. The weight of a path of length zero is by definition $1$.

If $\mathcal{L}$ is a finite collection of lattice paths, the expression $\sum_{\gamma\in\mathcal{L}}w(\gamma)$ will be called the \textit{weight polynomial associated with $\mathcal{L}$}. If $\mathcal{L}$ is the empty collection, its weight polynomial is understood to be zero. 

Given $n\in\mathbb{Z}_{\geq 0}$ and $j\in\{0,\ldots,p\}$, let $\mathcal{P}_{[n,j]}$ denote the collection of all paths of length $n$, with starting point $(0,0)$ and final point $(n,j)$. We also define $\mathcal{D}_{[n,j]}$ as the collection of all paths $\gamma$ of length $n$, with initial point $(0,0)$, final point $(n,j)$, and satisfying $\min(\gamma)=0$. So $\mathcal{D}_{[n,j]}$ is the subcollection of $\mathcal{P}_{[n,j]}$ consisting of those paths with no point below the $x$-axis. We shall use the notations $\mathcal{P}_{n}=\mathcal{P}_{[n,0]}$ and $\mathcal{D}_{n}=\mathcal{D}_{[n,0]}$. In Figs.~\ref{Lukapath}--\ref{bridge} we illustrate some examples of paths in the collections defined.

We need another class of lattice paths. For each $n\in\mathbb{Z}_{\geq 0}$ and $j\in\{0,\ldots,p\}$, let $\widehat{\mathcal{D}}_{[n,j]}$ denote the collection of all paths $\gamma$ of length $n$, with initial point $(0,-j)$, final point $(n,0)$, and satisfying $\max(\gamma)=0$. Note that $\widehat{\mathcal{D}}_{[n,0]}$ is a subcollection of $\mathcal{P}_{[n,0]}$, but $\widehat{\mathcal{D}}_{[n,j]}$ is not a subcollection of $\mathcal{P}_{[n,j]}$ if $j\geq 1$. There is a one-to-one correspondence between the sets $\mathcal{D}_{[n,j]}$ and $\widehat{\mathcal{D}}_{[n,j]}$, established by the map $\gamma\mapsto\widehat{\gamma}$ that is defined as follows: given a path $\gamma\in\mathcal{D}_{[n,j]}$, it is first reflected with respect to the real axis, and the result is then reflected with respect to the vertical line $x=n/2$, to obtain the path $\widehat{\gamma}\in\widehat{\mathcal{D}}_{[n,j]}$. Note that under this transformation, the image of a step that belongs to one of the sets in \eqref{defsteps} is a step that belongs to the same set.

In the theory of lattice paths, different terminologies are used to refer to the paths in the collections $\mathcal{D}_{[n,j]}$ and $\mathcal{P}_{[n,j]}$. In the recent works \cite{PetSok,PetSokZhu}, the authors use the terms \textit{$p$-Lukasiewicz paths} and \textit{partial $p$-Lukasiewicz paths} to refer to the paths in $\mathcal{D}_{n}$ and $\mathcal{D}_{[n,j]}$, $j>0$, respectively. In \cite{BanFla}, the paths in $\mathcal{P}_{n}$, $\mathcal{P}_{[n,j]}$, $\mathcal{D}_{n}$, and $\mathcal{D}_{[n,j]}$, are generically called \textit{bridges}, \textit{walks}, \textit{excursions}, and \textit{meanders}, respectively.

The weight polynomials associated with the collections $\mathcal{P}_{[n,j]}$, $\mathcal{D}_{[n,j]}$, and $\widehat{\mathcal{D}}_{[n,j]}$ are denoted
\begin{align}
W_{[n,j]} & :=\sum_{\gamma\in\mathcal{P}_{[n,j]}}w(\gamma),\label{eq:def:Wnj}\\
A_{[n,j]} & :=\sum_{\gamma\in\mathcal{D}_{[n,j]}}w(\gamma),\label{eq:def:Anj}\\
B_{[n,j]} & :=\sum_{\gamma\in\widehat{\mathcal{D}}_{[n,j]}}w(\gamma).\label{eq:def:Bnj}
\end{align}

In our analysis we will also need the following polynomials. For an integer $q\geq 0$, let
\begin{align}
A_{[n,j]}^{(q)} & :=\sum_{\gamma\in\mathcal{D}_{[n,j]}}w(\gamma+q),\label{eq:def:Anjshifted}\\
B_{[n,j]}^{(q)} & :=\sum_{\gamma\in\widehat{\mathcal{D}}_{[n,j]}}w(\gamma-q).\label{eq:def:Bnjq}
\end{align}
Thus, $A_{[n,j]}^{(q)}$ and $B_{[n,j]}^{(q)}$ are the weight polynomials associated with the collections of paths
\begin{align}
\mathcal{D}_{[n,j]}^{(q)} & :=\{\gamma+q: \gamma\in\mathcal{D}_{[n,j]}\}\qquad q\in\mathbb{Z}_{\geq 0},\label{def:Dnjq}\\
\widehat{\mathcal{D}}_{[n,j]}^{(q)} & :=\{\gamma-q: \gamma\in\widehat{\mathcal{D}}_{[n,j]}\}\qquad q\in\mathbb{Z}_{\geq 0}.\label{eq:defDnjqtilde}
\end{align}

It will be convenient for us to define
\begin{align}
A_{[n,j]}=B_{[n,j]}=W_{[n,j]} & =0\qquad n<0,\quad 0\leq j\leq p,\label{def:extABW}\\
A_{[n,j]}^{(q)}=B_{[n,j]}^{(q)} & =0\qquad n<0,\quad 0\leq j\leq p,\quad q\geq 0.\label{def:extABWq}
\end{align}

We introduce now formal power series generated by the sequences of weight polynomials defined. For each $0\leq j\leq p$, let
\begin{align}
W_{j}(z) & :=\sum_{n=0}^{\infty}\frac{W_{[n,j]}}{z^{n+1}}=\sum_{n\in\mathbb{Z}}\frac{W_{[n,j]}}{z^{n+1}}\label{eq:def:seriesWj}\\
A_{j}(z) & :=\sum_{n=0}^{\infty}\frac{A_{[n,j]}}{z^{n+1}}=\sum_{n\in\mathbb{Z}}\frac{A_{[n,j]}}{z^{n+1}} \label{eq:def:seriesAj}\\
B_{j}(z) & :=\sum_{n=0}^{\infty}\frac{B_{[n,j]}}{z^{n+1}}=\sum_{n\in\mathbb{Z}}\frac{B_{[n,j]}}{z^{n+1}}\label{eq:def:seriesBj}
\end{align}
and for an integer $q\geq 0$, we define
\begin{align}
A_{j}^{(q)}(z) & :=\sum_{n=0}^{\infty}\frac{A_{[n,j]}^{(q)}}{z^{n+1}}=\sum_{n\in\mathbb{Z}}\frac{A_{[n,j]}^{(q)}}{z^{n+1}}\label{eq:def:seriesAjq}\\
B_{j}^{(q)}(z) & :=\sum_{n=0}^{\infty}\frac{B_{[n,j]}^{(q)}}{z^{n+1}}=\sum_{n\in\mathbb{Z}}\frac{B_{[n,j]}^{(q)}}{z^{n+1}}.
\label{eq:def:seriesBjq}
\end{align}
Observe that if $0\leq n<j$, then $\mathcal{P}_{[n,j]}=\emptyset$, and so $W_{[n,j]}=0$. If $n=j$, then $\mathcal{P}_{[j,j]}$ contains only one path consisting of consecutive upsteps connecting the points $(0,0)$ and $(j,j)$, and so $W_{[j,j]}=1$. Therefore
\[
W_{j}(z)=\frac{1}{z^{j+1}}+O\left(\frac{1}{z^{j+2}}\right).
\]
We can say the same about the other formal power series defined in \eqref{eq:def:seriesAj}--\eqref{eq:def:seriesBjq}. Our first main result is the following.

\begin{theorem}\label{theo:relseriesAs}
The following relations hold between the series defined in \eqref{eq:def:seriesAj} and \eqref{eq:def:seriesAjq}:
\begin{align}
A_{0}(z) & =\frac{1}{z-a_{0}^{(0)}-\sum_{j=1}^{p}a_{0}^{(j)}\,A_{j-1}^{(1)}(z)}\label{eq:A0Aks}\\
A_{j}(z) & =A_{0}(z)\,A^{(1)}_{j-1}(z)\qquad 1\leq j\leq p.\label{eq:AjA0Ajm}
\end{align}
\end{theorem}

The identities \eqref{eq:A0Aks} and \eqref{eq:AjA0Ajm} for $1\leq j\leq p-1$ establish a relation between the vectors $(A_{0}(z),\ldots,A_{p-1}(z))$ and $(A_{0}^{(1)}(z),\ldots,A_{p-1}^{(1)}(z))$ that can be expressed using \eqref{divoper} as
\[
(A_{0}(z),\ldots,A_{p-1}(z))=\frac{(1,\ldots,1)}{(0,\ldots,0,z-a_{0}^{(0)})+(A_{0}^{(1)}(z),\ldots,A_{p-2}^{(1)}(z),-\sum_{j=1}^{p}a_{0}^{(j)} A_{j-1}^{(1)}(z))}.
\]
This is the first step in the construction of the vector continued fraction for $(A_{0}(z),\ldots,A_{p-1}(z))$. One can then iterate the relations \eqref{relA0kAkpone}, of the same type as \eqref{eq:A0Aks}--\eqref{eq:AjA0Ajm}, to obtain the formal identity 
\begin{equation}\label{vcfforAs}
(A_{0}(z),\ldots,A_{p-1}(z))=\Kont_{j = 1}^{\infty}\Bigl(\frac{\mathbf{c}_{j}}{\mathbf{d}_{j}(z)}\Bigr).
\end{equation}
We explain the process of deriving this continued fraction in Section~\ref{sec:resolvcf}. The process is not straightforward since the coefficients $\mathbf{c}_{k}$ and $\mathbf{d}_{k}(z)$ change their form at some point along the way in the expansion, see \eqref{cks} and \eqref{dks}.

In this paper we also prove the relations
\begin{align}
zA_{0}(z)-1 & =\sum_{j=0}^{p} a_{0}^{(j)} A_{j}(z),\label{eq:relAs}\\
A_{j}(z) & =A_{i}(z)\,A_{j-i-1}^{(i+1)}(z),\qquad 0\leq i<j\leq p,\label{eq:A0Ai:2}
\end{align}
which are used in the proof of Theorem~\ref{theo:relseriesAs}. Note that \eqref{eq:A0Ai:2} generalizes \eqref{eq:AjA0Ajm}.

In the recent works of P\'{e}tr\'{e}olle--Sokal \cite{PetSok} and P\'{e}tr\'{e}olle--Sokal--Zhu \cite{PetSokZhu}, Lukasiewicz paths have also been connected with branched continued fractions. These are scalar continued fractions that are obtained iterating the relation
\[
A_{0}^{(k)}(z)=\frac{1}{z-a_{k}^{(0)}-\sum_{j=1}^{p}a_{k}^{(j)}\prod_{\ell=1}^{j}A_{0}^{(k+\ell)}(z)},\qquad k\geq 0.
\]
Note that this relation can be deduced from \eqref{relA0kAkpone}. The work \cite{PetSokZhu} also contains an extensive study of combinatorial and total-positivity properties of the polynomials $A_{[n,j]}$ and other weight polynomials associated with lattice paths in the upper half-plane. 

Our second main result is the following. 

\begin{theorem}\label{theo:relseriesWs}
The following relations hold between the series defined in \eqref{eq:def:seriesWj}--\eqref{eq:def:seriesBjq}:
\begin{align}
W_{0}(z) &=\frac{1}{z-a_{0}^{(0)}-\sum_{j=1}^{p}\sum_{k=0}^{j}a_{-k}^{(j)}\,A_{j-k-1}^{(1)}(z)\,B_{k-1}^{(1)}(z)} \label{eq:W0expansion}\\
W_{j}(z) & =W_{0}(z)\,A_{j-1}^{(1)}(z)\qquad 1\leq j\leq p.\label{eq:WjW0Aj}
\end{align}
In \eqref{eq:W0expansion}, we understand that $A_{-1}^{(1)}(z)\equiv B_{-1}^{(1)}(z)\equiv 1$. More generally,
\begin{align}
W_{j}(z) & =W_{i}(z)\,A_{j-i-1}^{(i+1)}(z) \qquad 0\leq i<j\leq p.\label{eq:WjWiAj}
\end{align}
\end{theorem}

It turns out that the power series $W_{j}(z)$ are resolvent functions of a two-sided banded Hessenberg operator, with matrix representation obtained by extending the diagonal sequences in \eqref{Hmatrix} from $\mathbb{Z}_{\geq 0}$ to $\mathbb{Z}$, see Proposition~\ref{prop:eqresollat}. So the relations in Theorem 1.4 are relations between resolvents of two-sided and one-sided operators. These relations played a crucial role in our study in \cite{LopPro,LopPro2} of characteristic polynomials of random banded Hessenberg matrices with i.i.d. entries along diagonals. 

In Section~\ref{sec:proofmain} we prove Theorems~\ref{theo:relseriesAs} and \ref{theo:relseriesWs}, using only notions from the theory of lattice paths. In Section~\ref{sec:STDyck} we focus on certain subcollections of the families $\mathcal{P}_{[n,j]}$, $\mathcal{D}_{[n,j]}$, and $\widehat{\mathcal{D}}_{[n,j]}$. The paths that we consider in that section are only allowed to have two types of steps:
\begin{equation}\label{SRedges}
\begin{aligned}
& \mbox{upsteps}\,\,\,(n,m)\rightarrow(n+1,m+1),\\
& \mbox{downsteps}\,\,\,(n,m)\rightarrow(n+1,m-p)\,\,\,\mbox{by $p$ units}.
\end{aligned}
\end{equation}
These steps are given the same weights specified in \eqref{weightedges}. Since $(a_{n}^{(p)})_{n\in\mathbb{Z}}$ is the only sequence of weights that will be relevant in this case, in that section we simplify notation and write 
\begin{equation}\label{eq:ananp}
a_{n}=a_{n}^{(p)},\qquad n\in\mathbb{Z}.
\end{equation}

For each $0\leq j\leq p$ and $n\in\mathbb{Z}_{\geq 0}$, let $\mathcal{R}_{[n,j]}$ denote the collection of all lattice paths of length $n$, with initial point $(0,0)$, final point $(n,j)$, and steps as indicated in \eqref{SRedges}. Hence $\mathcal{R}_{[n,j]}$ is a subcollection of $\mathcal{P}_{[n,j]}$. The difference in height between the initial and final points of a path $\gamma\in\mathcal{R}_{[n,j]}$ implies
\begin{equation}\label{relcountupdown}
\mbox{card}\,\{\mbox{upsteps in}\,\gamma\}=p\times \mbox{card}\,\{\mbox{downsteps in}\,\gamma\}+j.
\end{equation}
From this it is easy to deduce that 
\[
\mathcal{R}_{[n,j]}\neq \emptyset\,\,\Longleftrightarrow\,\, n\equiv j\,\,\mbox{mod}\,\,(p+1).
\]  
For $0\leq j\leq p$ and $n\in\mathbb{Z}_{\geq 0}$ we also define $\mathcal{S}_{[n,j]}:=\{\gamma\in\mathcal{R}_{[n,j]}: \min(\gamma)=0\}$. The paths in $\mathcal{S}_{[n,j]}$ are known as \emph{partial $p$-Dyck paths}. Finally, let $\widehat{\mathcal{S}}_{[n,j]}$ denote the collection of all paths $\gamma$ of length $n$, with initial point $(0,-j)$, terminal point $(n,0)$, steps as indicated in \eqref{SRedges}, and satisfying $\max(\gamma)=0$. The collections $\mathcal{S}_{[n,j]}$ and $\widehat{\mathcal{S}}_{[n,j]}$ are also non-empty if and only if $n-j$ is a multiple of $p+1$. In Figs.~\ref{genDyckpath}--\ref{reflectedDyckpath} we illustrate some paths in the collections we have defined.

We introduce now the associated families of weight polynomials. Let
\begin{align}
R_{[n,j]} & :=\sum_{\gamma\in\mathcal{R}_{[n,j]}}w(\gamma),\label{def:wpolyRnj}\\
S_{[n,j]} & :=\sum_{\gamma\in\mathcal{S}_{[n,j]}}w(\gamma),\label{def:wpolySnj}\\
T_{[n,j]} & :=\sum_{\gamma\in\widehat{\mathcal{S}}_{[n,j]}}w(\gamma).\label{def:wpolyTnj}
\end{align}
In particular, we have 
\[
R_{[n,j]}=S_{[n,j]}=T_{[n,j]}=0\quad  \mbox{if}\,\, n\not\equiv j\,\,\mbox{mod}\,\,(p+1).
\]
The polynomials $S_{[n,j]}$ are referred in \cite{PetSokZhu} as \emph{generalized $p$-Stieltjes--Rogers polynomials}. 

Our interest in the analysis of the subcollections $\mathcal{R}_{[n,j]}$, $\mathcal{S}_{[n,j]}$, $\widehat{\mathcal{S}}_{[n,j]}$ comes from the work of Aptekarev--Kaliaguine--Van Iseghem \cite{AptKalVan}, where they study spectral properties of one-sided banded Hessenberg operators $H$ in the bi-diagonal case, i.e., taking $a_{n}^{(k)}=0$, $0\leq k\leq p-1$ and $a_{n}^{(p)}=a_{n}$ in \eqref{Hmatrix}. In that work, formula \eqref{eq:descSnj} below was obtained for the moments $\langle H^{n}e_{j},e_{0}\rangle$, $n=m(p+1)+j$, which are called \emph{genetic sums}. Under the assumption that the sequence $(a_{n})_{n=0}^{\infty}$ is bounded and its terms are positive, it was also proved in \cite{AptKalVan} that the Hermite-Pad\'{e} approximants \eqref{vectHPapprox} converge to the vector of resolvent functions $(\phi_{0}(z),\ldots,\phi_{p-1}(z))=(S_{0}(z),\ldots,S_{p-1}(z))$ for $z$ in the complement of the starlike set $\{x\in\mathbb{C}: x^{p+1}\geq 0\}$ in the complex plane, where
\[
S_{j}(z)=\sum_{m=0}^{\infty}\frac{S_{[m(p+1)+j,j]}}{z^{m(p+1)+j+1}},\qquad 0\leq j\leq p-1.
\]

In Section~\ref{sec:STDyck} our main result is the following theorem, which we prove again using only the theory of lattice paths.

\begin{theorem}\label{theo:weightpoly}
For every $m\in\mathbb{Z}_{\geq 0}$ and $0\leq j\leq p$ the following identities hold:
\begin{align}
R_{[m(p+1)+j,j]} & =\sum_{i_{1}=-p}^{(m-1)p+j}\,\sum_{i_{2}=i_{1}-p}^{(m-2)p+j}\,\sum_{i_{3}=i_{2}-p}^{(m-3)p+j}\,\cdots\,\sum_{i_{m}=i_{m-1}-p}^{j}\,\prod_{k=1}^{m}a_{i_{k}},\label{eq:descRnj}\\
S_{[m(p+1)+j,j]} & =\sum_{i_{1}=0}^{j}\,\,\sum_{i_{2}=0}^{i_{1}+p}\,\,\sum_{i_{3}=0}^{i_{2}+p}\,\,\cdots\,\,\sum_{i_{m}=0}^{i_{m-1}+p}\,\,\prod_{k=1}^{m}a_{i_{k}},\label{eq:descSnj}\\
T_{[m(p+1)+j,j]} & =\sum_{i_{1}=-j-p}^{-p}\,\,\sum_{i_{2}=i_{1}-p}^{-p}\,\,\sum_{i_{3}=i_{2}-p}^{-p}\,\,\cdots\,\,\sum_{i_{m}=i_{m-1}-p}^{-p}\,\,\prod_{k=1}^{m}a_{i_{k}}.\label{eq:descTnj}
\end{align} 
These expressions are understood to be $1$ if $m=0$.
\end{theorem}

In Section~\ref{sec:resolvcf} we first show that the formal power series defined in \eqref{eq:def:seriesWj}--\eqref{eq:def:seriesBjq} can be identified as resolvent functions of certain banded Hessenberg operators, and in particular we prove that $(A_{0}(z),\ldots,A_{p-1}(z))=(\phi_{0}(z),\ldots,\phi_{p-1}(z))$. In the second part of that section we justify \eqref{vcfforAs} using the relations \eqref{eq:A0Aks}--\eqref{eq:AjA0Ajm} and \eqref{relA0kAkpone}. We also describe two different vector continued fractions for the vector $(S_{0}(z),\ldots,S_{p-1}(z))$ in the bi-diagonal case.

\section{Proofs of Theorems~\ref{theo:relseriesAs} and \ref{theo:relseriesWs}}\label{sec:proofmain}

\noindent\textbf{Proof of Theorem~\ref{theo:relseriesAs} and formulas \eqref{eq:relAs}--\eqref{eq:A0Ai:2}:} First we prove \eqref{eq:AjA0Ajm}, so let $1\leq j\leq p$ be fixed. It is clear that if $0\leq n\leq j-1$, then $\mathcal{D}_{[n,j]}=\emptyset$ and $A_{[n,j]}=0$. Let $n\geq j$ and $\gamma$ be a path in $\mathcal{D}_{[n,j]}$. We can find in $\gamma$ a unique upstep with the following property: it is the last upstep of the form $(\kappa,0)\rightarrow (\kappa+1,1)$ (i.e. with initial height $0$ and final height $1$) as we traverse the path from left to right. This is obvious since the path starts at height $0$ and ends at height $j\geq 1$. Let us denote the abscissa of the initial point of this special upstep with the symbol $\kappa_{0}(\gamma)$.

For $n\geq j$, we can partition the collection $\mathcal{D}_{[n,j]}$ as follows:
\begin{equation}\label{eq:partDnj}
\mathcal{D}_{[n,j]}=\bigcup_{k=0}^{n-j}\mathcal{D}_{[n,j,k]},
\end{equation}
where 
\begin{equation}\label{def:Dnjk}
\mathcal{D}_{[n,j,k]}:=\{\gamma\in\mathcal{D}_{[n,j]}: \kappa_{0}(\gamma)=k\},\qquad 0\leq k\leq n-j.
\end{equation}
It is very easy to see that the collections $\mathcal{D}_{[n,j,k]}$, $0\leq k\leq n-j$ are disjoint.

Let $\gamma\in\mathcal{D}_{[n,j,k]}$. We subdivide $\gamma$ into three parts $\gamma_{1}$, $\gamma_{2}$, $\gamma_{3}$, which are the following: $\gamma_{1}$ is the portion of $\gamma$ on the interval $0\leq x\leq k$, $\gamma_{2}$ is the single special upstep $(k,0)\rightarrow(k+1,1)$, and $\gamma_{3}$ is the portion of $\gamma$ on the interval $k+1\leq x\leq n$. The first part $\gamma_{1}$ is clearly identifiable with a path in $\mathcal{D}_{[k,0]}$, because $\gamma_{1}$ has length $k$, and it starts and ends on the real axis. It follows from the definition of $\kappa_{0}(\gamma)=k$ that $\gamma_{3}$ has no point below the horizontal line $y=1$, and therefore it can be identified with a horizontal translation of a path in $\mathcal{D}^{(1)}_{[n-k-1,j-1]}$. Since $w(\gamma)=w(\gamma_{1}) w(\gamma_{2}) w(\gamma_{3})=w(\gamma_{1}) w(\gamma_{3})$, these identifications allow us to conclude that
\begin{equation}\label{eq:weightDnjk}
\sum_{\gamma\in\mathcal{D}_{[n,j,k]}}w(\gamma)=A_{[k,0]}\,A_{[n-k-1,j-1]}^{(1)}.
\end{equation}
Applying now \eqref{eq:def:Anj}, \eqref{eq:partDnj}, and \eqref{eq:weightDnjk}, we obtain
\[
A_{[n,j]}=\sum_{\gamma\in\mathcal{D}_{[n,j]}}w(\gamma)=\sum_{k=0}^{n-j}\sum_{\gamma\in\mathcal{D}_{[n,j,k]}}w(\gamma)=\sum_{k=0}^{n-j} A_{[k,0]}\,A_{[n-k-1,j-1]}^{(1)},\qquad n\geq j.
\]
In virtue of \eqref{def:extABW}, \eqref{def:extABWq}, and the fact that $A_{[l,j-1]}^{(1)}=0$ for $0\leq l\leq j-2$, we can rewrite the above identity as
\[
A_{[n,j]}=\sum_{k\in\mathbb{Z}} A_{[k,0]} A_{[n-k-1,j-1]}^{(1)},\qquad n\in\mathbb{Z}.
\]
Hence, from \eqref{eq:def:seriesAj} and \eqref{eq:def:seriesAjq} we deduce 
\begin{equation}\label{rel3As}
A_{j}(z)=\sum_{n\in\mathbb{Z}}\frac{A_{[n,j]}}{z^{n+1}}=\sum_{n\in\mathbb{Z}}\frac{1}{z^{n+1}}\sum_{k\in\mathbb{Z}}A_{[k,0]} A_{[n-k-1,j-1]}^{(1)}=A_{0}(z) A_{j-1}^{(1)}(z).
\end{equation}

\begin{figure}
\begin{center}
\begin{tikzpicture}[scale=0.7]
\draw[line width=1.5pt]  (-3.03,0) -- (15.5,0);
\draw[line width=1.5pt]  (-3,0) -- (-3,4.5);
\draw[line width=0.7pt] (-3,0) -- (-2,1) -- (-1,1) -- (0,2) -- (1,1) -- (2,2) -- (3,3) -- (4,3) -- (5,4) -- (6,1) -- (7,2) -- (8,3) -- (9,2) -- (10,2) -- (11,0) -- (12,1) -- (13,1) -- (14,2) -- (15,3);
\draw [line width=0.5] (-2,0) -- (-2,-0.12);
\draw [line width=0.5] (-1,0) -- (-1,-0.12);
\draw [line width=0.5] (0,0) -- (0,-0.12);
\draw [line width=0.5] (1,0) -- (1,-0.12);
\draw [line width=0.5] (2,0) -- (2,-0.12);
\draw [line width=0.5] (3,0) -- (3,-0.12);
\draw [line width=0.5] (4,0) -- (4,-0.12);
\draw [line width=0.5] (5,0) -- (5,-0.12);
\draw [line width=0.5] (6,0) -- (6,-0.12);
\draw [line width=0.5] (7,0) -- (7,-0.12);
\draw [line width=0.5] (8,0) -- (8,-0.12);
\draw [line width=0.5] (9,0) -- (9,-0.12);
\draw [line width=0.5] (10,0) -- (10,-0.12);
\draw [line width=0.5] (11,0) -- (11,-0.12);
\draw [line width=0.5] (12,0) -- (12,-0.12);
\draw [line width=0.5] (13,0) -- (13,-0.12);
\draw [line width=0.5] (14,0) -- (14,-0.12);
\draw [line width=0.5] (15,0) -- (15,-0.12);
\draw [line width=0.5] (-3,1) -- (-3.12,1);
\draw [line width=0.5] (-3,2) -- (-3.12,2);
\draw [line width=0.5] (-3,3) -- (-3.12,3);
\draw [line width=0.5] (-3,4) -- (-3.12,4);
\draw [dotted] (-3,1) -- (15.5,1);
\draw [dotted] (-3,2) -- (15.5,2);
\draw [dotted] (-3,3) -- (15.5,3);
\draw [dotted] (-3,4) -- (15.5,4);
\draw [dotted] (-2,0) -- (-2,4.5);
\draw [dotted] (-1,0) -- (-1,4.5);
\draw [dotted] (0,0) -- (0,4.5);
\draw [dotted] (1,0) -- (1,4.5);
\draw [dotted] (2,0) -- (2,4.5);
\draw [dotted] (3,0) -- (3,4.5);
\draw [dotted] (4,0) -- (4,4.5);
\draw [dotted] (5,0) -- (5,4.5);
\draw [dotted] (6,0) -- (6,4.5);
\draw [dotted] (7,0) -- (7,4.5);
\draw [dotted] (8,0) -- (8,4.5);
\draw [dotted] (9,0) -- (9,4.5);
\draw [dotted] (10,0) -- (10,4.5);
\draw [dotted] (11,0) -- (11,4.5);
\draw [dotted] (12,0) -- (12,4.5);
\draw [dotted] (13,0) -- (13,4.5);
\draw [dotted] (14,0) -- (14,4.5);
\draw [dotted] (15,0) -- (15,4.5);
\draw (-2,-0.1) node[below, scale=0.8]{$1$};
\draw (-1,-0.1) node[below, scale=0.8]{$2$};
\draw (0,-0.1) node[below, scale=0.8]{$3$};
\draw (1,-0.1) node[below, scale=0.8]{$4$};
\draw (2,-0.1) node[below, scale=0.8]{$5$};
\draw (3,-0.1) node[below, scale=0.8]{$6$};
\draw (4,-0.1) node[below, scale=0.8]{$7$};
\draw (5,-0.1) node[below, scale=0.8]{$8$};
\draw (6,-0.1) node[below, scale=0.8]{$9$};
\draw (7,-0.1) node[below, scale=0.8]{$10$};
\draw (8,-0.1) node[below, scale=0.8]{$11$};
\draw (9,-0.1) node[below, scale=0.8]{$12$};
\draw (10,-0.1) node[below, scale=0.8]{$13$};
\draw (11,-0.1) node[below, scale=0.8]{$14$};
\draw (12,-0.1) node[below, scale=0.8]{$15$};
\draw (13,-0.1) node[below, scale=0.8]{$16$};
\draw (14,-0.1) node[below, scale=0.8]{$17$};
\draw (15,-0.1) node[below, scale=0.8]{$18$};
\draw (-3.1,1) node[left, scale=0.8]{$1$};
\draw (-3.1,2) node[left, scale=0.8]{$2$};
\draw (-3.1,3) node[left, scale=0.8]{$3$};
\draw (-3.1,4) node[left, scale=0.8]{$4$};
\end{tikzpicture}
\end{center}
\caption{Example, in the case $p=3$, of a path in the collection $\mathcal{D}_{[18,3]}$.}
\label{Lukapath}
\end{figure}
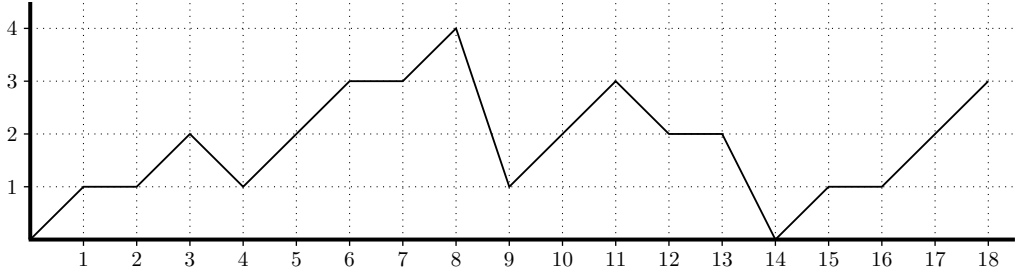

Now we justify \eqref{eq:A0Ai:2}, so fix $0\leq i<j\leq p$. Let $\gamma\in\mathcal{D}_{[n,j]}$, $n\geq j$. There exists a unique upstep in $\gamma$ with the property that it is the last upstep of the form $(\kappa,i)\rightarrow(\kappa+1,i+1)$ (i.e., with initial height $i$ and final height $i+1$) as we traverse the path from left to right. This is clear since $\gamma$ starts at height $0$, it ends at height $j$, and $j>i$. Let $\kappa_{i}(\gamma)$ denote the abscissa of the initial point of this unique upstep in $\gamma$. A simple counting shows that the possible range of the values of $\kappa_{i}(\gamma)$ is $i\leq \kappa_{i}(\gamma)\leq n-j+i$. So we partition the collection $\mathcal{D}_{[n,j]}$ as follows:
\begin{equation}\label{eq:partDnjDtilde}
\mathcal{D}_{[n,j]}=\bigcup_{k=i}^{n-j+i}\widetilde{\mathcal{D}}_{[n,j,k]},
\end{equation} 
where
\[
\widetilde{\mathcal{D}}_{[n,j,k]}=\{\gamma\in\mathcal{D}_{[n,j]}: \kappa_{i}(\gamma)=k\},\qquad i\leq k\leq n-j+i.
\]
Let $\gamma\in\widetilde{\mathcal{D}}_{[n,j,k]}$. We subdivide $\gamma$ into three parts $\gamma_{1}$, $\gamma_{2}$, $\gamma_{3}$, which are the following: $\gamma_{1}$ is the portion of $\gamma$ on the interval $0\leq x\leq k$, which is clearly identifiable with a path in $\mathcal{D}_{[k,i]}$; $\gamma_{2}$ is the special upstep $(k,i)\rightarrow(k+1,i+1)$; and $\gamma_{3}$ is the portion of $\gamma$ on the interval $k+1\leq x\leq n$, which (in virtue of the definition of $\kappa_{i}(\gamma)$) has no point below the horizontal line $y=i+1$, and therefore it can be identified with a horizontal translation of a path in $\mathcal{D}_{[n-k-1,j-i-1]}^{(i+1)}$. 

These identifications imply the relation
\[
\sum_{\gamma\in\widetilde{\mathcal{D}}_{[n,j,k]}}w(\gamma)=A_{[k,i]} A_{[n-k-1,j-i-1]}^{(i+1)}.
\]
This relation and \eqref{eq:partDnjDtilde} lead to the identity    
\[
A_{[n,j]}=\sum_{k=i}^{n-j+i}\sum_{\gamma\in\widetilde{\mathcal{D}}_{[n,j,k]}}w(\gamma)=\sum_{k=i}^{n-j+i}A_{[k,i]} A_{[n-k-1,j-i-1]}^{(i+1)},\qquad n\geq j,
\]
which is equivalent to 
\[
A_{[n,j]}=\sum_{k\in\mathbb{Z}}A_{[k,i]} A_{[n-k-1,j-i-1]}^{(i+1)},\qquad n\in\mathbb{Z}.
\]
As in \eqref{rel3As}, this implies \eqref{eq:A0Ai:2}.

Now we justify \eqref{eq:relAs}. Suppose $n\geq 1$. If $\gamma\in\mathcal{D}_{[n,0]}$, then the last step of $\gamma$ is either the level step $(n-1,0)\rightarrow(n,0)$, or a downstep $(n-1,j)\rightarrow(n,0)$ for some $1\leq j\leq p$. We partition the collection of paths $\mathcal{D}_{[n,0]}$ according to the last step of a path $\gamma\in\mathcal{D}_{[n,0]}$. We have
\begin{equation}\label{eq:partDn0}
\mathcal{D}_{[n,0]}=\bigcup_{j=0}^{p}\mathcal{L}_{[n,j]}
\end{equation}
where
\[
\mathcal{L}_{[n,j]}:=\{\gamma\in\mathcal{D}_{[n,0]}: \mbox{the last step of}\,\,\gamma\,\,\mbox{is}\,\,(n-1,j)\rightarrow(n,0)\},\qquad 0\leq j\leq p.
\]
The collections $\mathcal{L}_{[n,j]}$ are obviously disjoint (some of them may be empty if $n\leq p$). If $\gamma\in\mathcal{L}_{[n,j]}$, the portion of $\gamma$ on the interval $0\leq x\leq n-1$ can be any path in $\mathcal{D}_{[n-1,j]}$, hence
\begin{equation}\label{eq:weightLnj}
\sum_{\gamma\in\mathcal{L}_{[n,j]}}w(\gamma)=a_{0}^{(j)} A_{[n-1,j]},\qquad 0\leq j\leq p.
\end{equation}
It follows from \eqref{eq:partDn0} and \eqref{eq:weightLnj} that 
\[
A_{[n,0]}=\sum_{\gamma\in\mathcal{D}_{[n,0]}} w(\gamma)=\sum_{j=0}^{p}\sum_{\gamma\in\mathcal{L}_{[n,j]}}w(\gamma)=\sum_{j=0}^{p} a_{0}^{(j)}A_{[n-1,j]},\qquad n\geq 1.
\]
This is equivalent to \eqref{eq:relAs}.

From \eqref{eq:relAs} and \eqref{eq:AjA0Ajm} we obtain
\begin{align*}
z A_{0}(z)-1 & = a_{0}^{(0)} A_{0}(z)+\sum_{j=1}^{p} a_{0}^{(j)} A_{j}(z)=a_{0}^{(0)} A_{0}(z)+\sum_{j=1}^{p} a_{0}^{(j)} A_{0}(z) A_{j-1}^{(1)}(z)\\
 & =A_{0}(z) (a_{0}^{(0)}+\sum_{j=1}^{p} a_{0}^{(j)} A_{j-1}^{(1)}(z))
\end{align*}
and \eqref{eq:A0Aks} follows.\qed

\begin{figure}
\begin{center}
\begin{tikzpicture}[scale=0.7]
\draw[line width=1.5pt]  (-3,0) -- (15.5,0);
\draw[line width=1.5pt]  (-3,-2.5) -- (-3,3.5);
\draw[line width=0.7pt] (-3,0) -- (-2,1) -- (-1,-1) -- (0,-1) -- (1,0) -- (2,-2) -- (3,-1) -- (4,0) -- (5,1) -- (6,0) -- (7,1) -- (8,1) -- (9,2) -- (10,3) -- (11,0) -- (12,1) -- (13,1) -- (14,-1) -- (15,0);
\draw [line width=0.5] (-2,0) -- (-2,-0.12);
\draw [line width=0.5] (-1,0) -- (-1,-0.12);
\draw [line width=0.5] (0,0) -- (0,-0.12);
\draw [line width=0.5] (1,0) -- (1,-0.12);
\draw [line width=0.5] (2,0) -- (2,-0.12);
\draw [line width=0.5] (3,0) -- (3,-0.12);
\draw [line width=0.5] (4,0) -- (4,-0.12);
\draw [line width=0.5] (5,0) -- (5,-0.12);
\draw [line width=0.5] (6,0) -- (6,-0.12);
\draw [line width=0.5] (7,0) -- (7,-0.12);
\draw [line width=0.5] (8,0) -- (8,-0.12);
\draw [line width=0.5] (9,0) -- (9,-0.12);
\draw [line width=0.5] (10,0) -- (10,-0.12);
\draw [line width=0.5] (11,0) -- (11,-0.12);
\draw [line width=0.5] (12,0) -- (12,-0.12);
\draw [line width=0.5] (13,0) -- (13,-0.12);
\draw [line width=0.5] (14,0) -- (14,-0.12);
\draw [line width=0.5] (15,0) -- (15,-0.12);
\draw [line width=0.5] (-3,1) -- (-3.12,1);
\draw [line width=0.5] (-3,2) -- (-3.12,2);
\draw [line width=0.5] (-3,3) -- (-3.12,3);
\draw [line width=0.5] (-3,0) -- (-3.12,0);
\draw [line width=0.5] (-3,-1) -- (-3.12,-1);
\draw [line width=0.5] (-3,-2) -- (-3.12,-2);
\draw [dotted] (-3,1) -- (15.5,1);
\draw [dotted] (-3,2) -- (15.5,2);
\draw [dotted] (-3,3) -- (15.5,3);
\draw [dotted] (-3,-1) -- (15.5,-1);
\draw [dotted] (-3,-2) -- (15.5,-2);
\draw [dotted] (-2,-2.5) -- (-2,3.5);
\draw [dotted] (-1,-2.5) -- (-1,3.5);
\draw [dotted] (0,-2.5) -- (0,3.5);
\draw [dotted] (1,-2.5) -- (1,3.5);
\draw [dotted] (2,-2.5) -- (2,3.5);
\draw [dotted] (3,-2.5) -- (3,3.5);
\draw [dotted] (4,-2.5) -- (4,3.5);
\draw [dotted] (5,-2.5) -- (5,3.5);
\draw [dotted] (6,-2.5) -- (6,3.5);
\draw [dotted] (7,-2.5) -- (7,3.5);
\draw [dotted] (8,-2.5) -- (8,3.5);
\draw [dotted] (9,-2.5) -- (9,3.5);
\draw [dotted] (10,-2.5) -- (10,3.5);
\draw [dotted] (11,-2.5) -- (11,3.5);
\draw [dotted] (12,-2.5) -- (12,3.5);
\draw [dotted] (13,-2.5) -- (13,3.5);
\draw [dotted] (14,-2.5) -- (14,3.5);
\draw [dotted] (15,-2.5) -- (15,3.5);
\draw (-2,-0.1) node[below, scale=0.8]{$1$};
\draw (-1,-0.1) node[below, scale=0.8]{$2$};
\draw (0,-0.1) node[below, scale=0.8]{$3$};
\draw (1,-0.1) node[below, scale=0.8]{$4$};
\draw (2,-0.1) node[below, scale=0.8]{$5$};
\draw (3,-0.1) node[below, scale=0.8]{$6$};
\draw (4,-0.1) node[below, scale=0.8]{$7$};
\draw (5,-0.1) node[below, scale=0.8]{$8$};
\draw (6,-0.1) node[below, scale=0.8]{$9$};
\draw (7,-0.1) node[below, scale=0.8]{$10$};
\draw (8,-0.1) node[below, scale=0.8]{$11$};
\draw (9,-0.1) node[below, scale=0.8]{$12$};
\draw (10,-0.1) node[below, scale=0.8]{$13$};
\draw (11,-0.1) node[below, scale=0.8]{$14$};
\draw (12,-0.1) node[below, scale=0.8]{$15$};
\draw (13,-0.1) node[below, scale=0.8]{$16$};
\draw (14,-0.1) node[below, scale=0.8]{$17$};
\draw (15,-0.1) node[below, scale=0.8]{$18$};
\draw (-3.1,1) node[left, scale=0.8]{$1$};
\draw (-3.1,2) node[left, scale=0.8]{$2$};
\draw (-3.1,3) node[left, scale=0.8]{$3$};
\draw (-3.1,0) node[left, scale=0.8]{$0$};
\draw (-3.1,-1) node[left, scale=0.8]{$-1$};
\draw (-3.1,-2) node[left, scale=0.8]{$-2$};
\end{tikzpicture}
\end{center}
\caption{Example, in the case $p=3$, of a path in the collection $\mathcal{P}_{[18,0]}$.}
\label{bridge}
\end{figure}
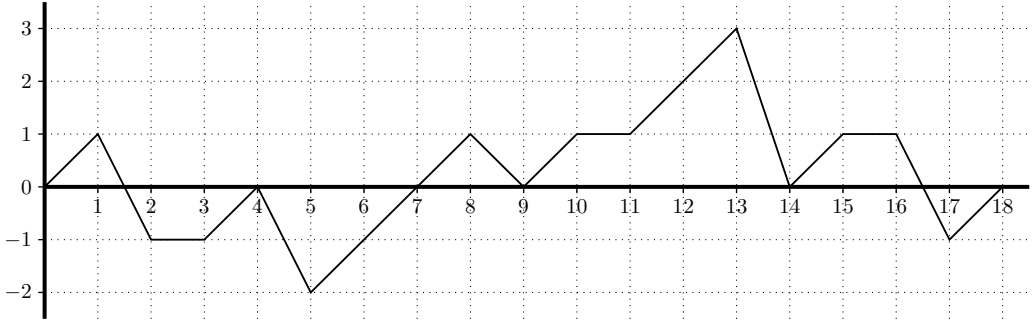

\begin{remark}
By analogy, it is clear that relations of the same kind as \eqref{eq:A0Aks} and \eqref{eq:AjA0Ajm} exist between the functions $A_{j}^{(k)}$ and $A_{j}^{(k+1)}$ for every integer $k\geq 0$. They are
\begin{equation}\label{relA0kAkpone}
\begin{aligned}
A_{0}^{(k)}(z) & =\frac{1}{z-a_{k}^{(0)}-\sum_{j=1}^{p}a_{k}^{(j)}\,A_{j-1}^{(k+1)}(z)}\\
A_{j}^{(k)}(z) & =A_{0}^{(k)}(z)\,A^{(k+1)}_{j-1}(z)\qquad 1\leq j\leq p.
\end{aligned}
\end{equation}
\end{remark}

\begin{remark}
The following relation is obtained from \eqref{eq:relAs} and \eqref{eq:A0Ai:2}: For each $0\leq i\leq p-1$,
\[
z A_{0}(z)-1=\sum_{j=0}^{i} a_{0}^{(j)}A_{j}(z)+\sum_{j=i+1}^{p} a_{0}^{(j)} A_{i}(z) A_{j-i-1}^{(i+1)}(z).
\]
In the case $i=0$ this reduces to \eqref{eq:A0Aks}.
\end{remark}

\smallskip

\noindent\textbf{Proof of Theorem~\ref{theo:relseriesWs}:} We first prove \eqref{eq:WjWiAj}. The argument is almost identical to the proof of \eqref{eq:A0Ai:2}, but for the convenience of the reader we reproduce it. Let $0\leq i<j\leq p$ be fixed, and let $n\geq j$. In any given path $\gamma\in\mathcal{P}_{[n,j]}$, we can find a unique upstep satisfying the following property: it is of the form $(\kappa,i)\rightarrow(\kappa,i+1)$ (i.e., with initial height $i$ and final height $i+1$) and it is the last upstep in $\gamma$ of this form as we traverse the path from left to right. The uniqueness is obvious and the existence of such a step is a consequence of the fact that $\gamma$ starts at height $0$ and ends at height $j>i$. As in the proof of \eqref{eq:A0Ai:2}, we denote by $\kappa_{i}(\gamma)$ the abscissa of the initial point of this special upstep in $\gamma$. The possible values of $\kappa_{i}(\gamma)$ are again $i\leq \kappa_{i}(\gamma)\leq n-j+i$.

We partition the collection $\mathcal{P}_{[n,j]}$ as follows:
\begin{equation}\label{eq:partPnj}
\mathcal{P}_{[n,j]}=\bigcup_{k=i}^{n-j+i}\mathcal{P}_{[n,j,k]}
\end{equation}
where
\[
\mathcal{P}_{[n,j,k]}=\{\gamma\in\mathcal{P}_{[n,j]}: \kappa_{i}(\gamma)=k\}\qquad i\leq k\leq n-j+i.
\]
The collections $\mathcal{P}_{[n,j,k]}$ are clearly disjoint, so \eqref{eq:partPnj} is indeed a partition. Given a path $\gamma\in\mathcal{P}_{[n,j,k]}$, we subdivide it into three parts $\gamma_{1}$, $\gamma_{2}$, $\gamma_{3}$, defined as follows: $\gamma_{1}$ is the portion of $\gamma$ on the interval $0\leq x\leq k$; $\gamma_{2}$ is the special upstep $(k,i)\rightarrow(k+1,i+1)$; and $\gamma_{3}$ is the remaining portion of $\gamma$ on the interval $k+1\leq x\leq n$. Then $\gamma_{1}$ is actually a path in $\mathcal{P}_{[k,i]}$. From the  definition of $\kappa_{i}(\gamma)$ we deduce that $\gamma_{3}$ has no point below the line $y=i+1$, and therefore $\gamma_{3}$ can be identified with a horizontal translation of a path in $\mathcal{D}_{[n-k-1,j-i-1]}^{(i+1)}$. From the relation $w(\gamma)=w(\gamma_{1})w(\gamma_{2})w(\gamma_{3})=w(\gamma_{1})w(\gamma_{3})$ and the above identifications of $\gamma_{1}$ and $\gamma_{3}$, we deduce that
\[
\sum_{\gamma\in\mathcal{P}_{[n,j,k]}}w(\gamma)=W_{[k,i]} A_{[n-k-1,j-i-1]}^{(i+1)}.
\]
This identity and \eqref{eq:partPnj} imply    
\[
W_{[n,j]}=\sum_{k=i}^{n-j+i}\sum_{\gamma\in\mathcal{P}_{[n,j,k]}}w(\gamma)=\sum_{k=i}^{n-j+i}W_{[k,i]} A_{[n-k-1,j-i-1]}^{(i+1)},\qquad n\geq j,
\]
and this immediately proves \eqref{eq:WjWiAj}. Formula \eqref{eq:WjW0Aj} is a particular case of \eqref{eq:WjWiAj}.

Now we justify \eqref{eq:W0expansion}. For $n\geq 1$, the collection $\mathcal{P}_{[n,0]}$ can be partitioned as follows:
\begin{equation}\label{eq:partPn0}
\mathcal{P}_{[n,0]}=\bigcup_{j=-1}^{p}\mathcal{U}_{[n,j]},\qquad n\geq 1,
\end{equation} 
where
\[
\mathcal{U}_{[n,j]}:=\{\gamma\in\mathcal{P}_{[n,0]}: \mbox{the first step in $\gamma$ is}\,\,(0,0)\rightarrow(1,-j)\},\qquad -1\leq j\leq p.
\]
The first step of a path in $\mathcal{U}_{[n,0]}$ is the level step $(0,0)\rightarrow(1,0)$, so it is clear that 
\begin{equation}\label{eq:weightUn0}
\sum_{\gamma\in\mathcal{U}_{[n,0]}}w(\gamma)=a_{0}^{(0)} W_{[n-1,0]}.
\end{equation}

We analyze now the paths in $\mathcal{U}_{[n,j]}$, $1\leq j\leq p$. In any $\gamma\in\mathcal{U}_{[n,j]}$, we can find a unique upstep satisfying the following condition: it is of the form $(k,-1)\rightarrow(k+1,0)$ (i.e., with initial height $-1$ and final height $0$) and it is the first upstep in $\gamma$ of this form as we traverse the path from left to right. The existence and uniqueness of such upstep is obvious since $\gamma$ ends at height $0$ and its initial step is a downstep.  Let us denote by $\lambda(\gamma)$ the abscissa of the initial point of this special upstep in $\gamma$. It is easy to see that the possible range of values of $\lambda(\gamma)$ is $j\leq \lambda(\gamma)\leq n-1$.

Then we have 
\begin{equation}\label{eq:Unjpartition}
\mathcal{U}_{[n,j]}=\bigcup_{k=j}^{n-1}\mathcal{U}_{[n,j,k]}
\end{equation}
where 
\begin{equation}\label{eq:def:Unjk}
\mathcal{U}_{[n,j,k]}:=\{\gamma\in\mathcal{U}_{[n,j]}: \lambda(\gamma)=k\},\qquad j\leq k\leq n-1.
\end{equation}
Obviously, the collections in \eqref{eq:def:Unjk} are pairwise disjoint. 

A path $\gamma\in\mathcal{U}_{[n,j,k]}$ can be subdivided into four parts $\gamma_{l}$, $1\leq l\leq 4$, as follows: $\gamma_{1}$ is the downstep $(0,0)\rightarrow(1,-j)$; $\gamma_{2}$ is the portion of $\gamma$ on the interval $1\leq x\leq k$; $\gamma_{3}$ is the special upstep $(k,-1)\rightarrow(k+1,0)$; and $\gamma_{4}$ is the remaining portion of $\gamma$ on the interval $k+1\leq x\leq n$. It follows from the definition of $\lambda(\gamma)$ that $\gamma_{2}$ has no point above the line $y=-1$. The form of $\gamma_{2}$ shows that it can be identified with a horizontal translation of a path in $\widehat{\mathcal{D}}_{[k-1,j-1]}^{(1)}$ (cf. \eqref{eq:defDnjqtilde}). The portion $\gamma_{4}$ is clearly a horizontal translation of a path in $\mathcal{P}_{[n-k-1,0]}$. We have $w(\gamma)=w(\gamma_{1})w(\gamma_{2})w(\gamma_{3})w(\gamma_{4})=a_{-j}^{(j)}\,w(\gamma_{2})w(\gamma_{4})$, and therefore from the identifications of $\gamma_{2}$ and $\gamma_{4}$ we conclude that
\[
\sum_{\gamma\in\mathcal{U}_{[n,j,k]}}w(\gamma)=a_{-j}^{(j)}\, B_{[k-1,j-1]}^{(1)} W_{[n-k-1,0]},
\]
cf. \eqref{eq:def:Bnjq} and \eqref{eq:def:Wnj}. This identity and \eqref{eq:Unjpartition} give the expression
\[
\sum_{\gamma\in\mathcal{U}_{[n,j]}}w(\gamma)=\sum_{k=j}^{n-1}\sum_{\gamma\in\mathcal{U}_{[n,j,k]}}w(\gamma)=\sum_{k=j}^{n-1} a_{-j}^{(j)}\,B_{[k-1,j-1]}^{(1)} W_{[n-k-1,0]}.
\]
This can also be written as
\begin{equation}\label{eq:weightpolyUnj}
\sum_{\gamma\in\mathcal{U}_{[n,j]}}w(\gamma)=
\sum_{k\in\mathbb{Z}}a_{-j}^{(j)}\,B_{[k-1,j-1]}^{(1)} W_{[n-k-1,0]},\qquad 1\leq j\leq p.
\end{equation}

Now we want to obtain an expression for the weight polynomial associated with the collection $\mathcal{U}_{[n,-1]}$. Recall that by definition, the first step in a path $\gamma\in\mathcal{U}_{[n,-1]}$ is the upstep $(0,0)\rightarrow(1,1)$. We will partition $\mathcal{U}_{[n,-1]}$ according to the \emph{position of the first downstep in $\gamma\in\mathcal{U}_{[n,-1]}$ that touches or crosses the real axis}. We say that \emph{a downstep touches the real axis} if its final point is on the axis, and \emph{a downstep crosses the real axis} if its initial point is above the axis and its final point is below the axis. It is clear that every path in $\mathcal{U}_{[n,-1]}$ has a downstep that touches or crosses the real axis, since the path has height $1$ at $x=1$ and it ends on the real axis. 

We define now the following subcollections of $\mathcal{U}_{[n,-1]}$. Let $\mathcal{V}_{[n,\ell]}$, $1\leq \ell\leq p$, denote the collection of all paths $\gamma\in\mathcal{U}_{[n,-1]}$ such that the first downstep in $\gamma$ that touches or crosses the real axis is of the form $(k,\ell)\rightarrow(k+1,0)$, i.e., it is a downstep that touches the real axis with initial height $\ell$. Let $\mathcal{C}_{[n,r,s]}$, $1\leq r\leq p-1$, $1\leq s\leq p-r$, denote the collection of all paths $\gamma\in\mathcal{U}_{[n,-1]}$ such that the first downstep in $\gamma$ that touches or crosses the real axis is of the form $(k,r)\rightarrow(k+1,-s)$, i.e., it is a downstep that crosses the real axis with initial height $r$ and final height $-s$. It is clear that $\mathcal{U}_{[n,-1]}$ is the disjoint union of all the subcollections $\mathcal{V}_{[n,\ell]}$ and $\mathcal{C}_{[n,r,s]}$. For a path in $\gamma\in\mathcal{U}_{[n,-1]}$, let $\eta(\gamma)$ denote the abscissa of the initial point of the first downstep in $\gamma$ that touches or crosses the real axis.

Let $1\leq \ell\leq p$. In the collection $\mathcal{V}_{[n,\ell]}$, the possible range of $\eta(\gamma)$ is clearly $\ell\leq \eta(\gamma)\leq n-1$. So we have the disjoint union
\[
\mathcal{V}_{[n,\ell]}=\bigcup_{k=\ell}^{n-1}\mathcal{V}_{[n,\ell,k]}
\]
where
\[
\mathcal{V}_{[n,\ell,k]}:=\{\gamma\in\mathcal{V}_{[n,\ell]}: \eta(\gamma)=k\},\qquad \ell\leq k\leq n-1.
\]
A path $\gamma$ in $\mathcal{V}_{[n,\ell,k]}$ is subdivided into four parts $\gamma_{j}$, $1\leq j\leq 4$. The first part $\gamma_{1}$ is the upstep $(0,0)\rightarrow(1,1)$, $\gamma_{2}$ is the portion of $\gamma$ on the interval $1\leq x\leq k$, $\gamma_{3}$ is the downstep $(k,\ell)\rightarrow(k+1,0)$, and $\gamma_{4}$ is the portion of $\gamma$ on the interval $k+1\leq x\leq n$. By definition of $\eta(\gamma)$, it is clear that $\gamma_{2}$ has no point below the line $y=1$, so $\gamma_{2}$ can be identified as a horizontal translation of a path in $\mathcal{D}_{[k-1,\ell-1]}^{(1)}$. The portion $\gamma_{4}$ is a horizontal translation of a path in $\mathcal{P}_{[n-k-1,0]}$. Since $w(\gamma)=w(\gamma_{1})w(\gamma_{2})w(\gamma_3)w(\gamma_4)=a_{0}^{(\ell)}w(\gamma_{2})w(\gamma_{4})$, we obtain
\[
\sum_{\gamma\in\mathcal{V}_{[n,\ell]}}w(\gamma)=\sum_{k=\ell}^{n-1}\sum_{\gamma\in\mathcal{V}_{[n,\ell,k]}}w(\gamma)=\sum_{k=\ell}^{n-1}a_{0}^{(\ell)} A^{(1)}_{[k-1,\ell-1]} W_{[n-k-1,0]}.
\]
So we have
\begin{equation}\label{eq:weightVnl}
\sum_{\gamma\in\mathcal{V}_{[n,\ell]}}w(\gamma)=\sum_{k\in\mathbb{Z}}a_{0}^{(\ell)} A^{(1)}_{[k-1,\ell-1]} W_{[n-k-1,0]},\qquad 1\leq \ell\leq p.
\end{equation}

Now we analyze the paths in the collection $\mathcal{C}_{[n,r,s]}$, $1\leq r\leq p-1$, $1\leq s\leq p-r$. It is easy to see that for the collection $\mathcal{C}_{[n,r,s]}$ to be non-empty it is necessary and sufficient that $r+s+1\leq n$. Let $\gamma$ be a path in $\mathcal{C}_{[n,r,s]}$. As before, $\lambda(\gamma)$ denotes the abscissa of the initial point of the first upstep in $\gamma$ from height $-1$ to height $0$ as we traverse the path from left to right. A simple counting shows that the possible range of values of $\eta(\gamma)$ is $r\leq \eta(\gamma)\leq n-s-1$. For each $\eta(\gamma)$ in this range, the possible values of $\lambda(\gamma)$ are $\eta(\gamma)+s\leq \lambda(\gamma)\leq n-1$.

Given a path $\gamma\in\mathcal{C}_{[n,r,s]}$, we subdivide it into six parts $\gamma_{j}$, $1\leq j\leq 6$, defined as follows: $\gamma_{1}$ is the initial upstep $(0,0)\rightarrow(1,1)$, $\gamma_{2}$ is the portion of $\gamma$ on the interval $1\leq x\leq \eta(\gamma)$ (which by definition of $\eta(\gamma)$ is above the line $y=1$), $\gamma_{3}$ is the special downstep $(\eta(\gamma),r)\rightarrow(\eta(\gamma)+1,-s)$ that crosses the real axis, $\gamma_{4}$ is the portion of $\gamma$ on the interval $\eta(\gamma)+1\leq x\leq \lambda(\gamma)$ (which by definition of $\lambda(\gamma)$ is below the line $y=-1$ and ends at height $-1$), $\gamma_{5}$ is the upstep $(\lambda(\gamma),-1)\rightarrow(\lambda(\gamma)+1,0)$, and $\gamma_{6}$ is the remaining portion of $\gamma$ on the interval $\lambda(\gamma)+1\leq x\leq n$. It is clear that $\gamma_{2}$ is a horizontal translation of a path in $\mathcal{D}_{[\eta(\gamma)-1,r-1]}^{(1)}$, $\gamma_{4}$ is a horizontal translation of a path in $\widehat{\mathcal{D}}_{[\lambda(\gamma)-\eta(\gamma)-1,s-1]}^{(1)}$, and $\gamma_{6}$ is a horizontal translation of a path in $\mathcal{P}_{[n-\lambda(\gamma)-1,0]}$. Since $w(\gamma_{1})=w(\gamma_{5})=1$ and $w(\gamma_{3})=a_{-s}^{(r+s)}$, we have $w(\gamma)=a_{-s}^{(r+s)} w(\gamma_{2}) w(\gamma_{4}) w(\gamma_{6})$. Using the indices $k=\eta(\gamma)$ and $\ell=\lambda(\gamma)$, the above identifications allow us to conclude that
\[
\sum_{\gamma\in\mathcal{C}_{[n,r,s]}}w(\gamma)=a_{-s}^{(r+s)}\sum_{k=r}^{n-s-1}\sum_{\ell=k+s}^{n-1} A_{[k-1,r-1]}^{(1)} B_{[\ell-k-1,s-1]}^{(1)}W_{[n-\ell-1,0]}.
\]
The reader can easily check that this is equivalently expressed as 
\begin{equation}\label{eq:weightCnrs}
\sum_{\gamma\in\mathcal{C}_{[n,r,s]}}w(\gamma)=a_{-s}^{(r+s)}\sum_{k\in\mathbb{Z}}\sum_{\ell\in\mathbb{Z}}A_{[k-1,r-1]}^{(1)} B_{[\ell-k-1,s-1]}^{(1)}W_{[n-\ell-1,0]}
\end{equation}
and each summation has finitely many non-zero terms.

Recall that $\mathcal{U}_{[n,-1]}$ is the disjoint union of the collections $\mathcal{V}_{[n,\ell]}$, $1\leq \ell\leq p$, and $\mathcal{C}_{[n,r,s]}$, $1\leq r\leq p-1$, $1\leq s\leq p-r$. It then follows from \eqref{eq:weightVnl} and \eqref{eq:weightCnrs} that
\begin{align}\label{eq:weightUnm1}
\sum_{\gamma\in\mathcal{U}_{[n,-1]}}w(\gamma) & =\sum_{\ell=1}^{p}\sum_{\gamma\in\mathcal{V}_{[n,\ell]}}w(\gamma)+\sum_{r=1}^{p-1}\sum_{s=1}^{p-r}\sum_{\gamma\in\mathcal{C}_{[n,r,s]}}w(\gamma)\notag\\
& =\sum_{\ell=1}^{p}\sum_{k\in\mathbb{Z}}a_{0}^{(\ell)} A^{(1)}_{[k-1,\ell-1]} W_{[n-k-1,0]}\notag\\
& +\sum_{r=1}^{p-1}\sum_{s=1}^{p-r}\sum_{k\in\mathbb{Z}}\sum_{\ell\in\mathbb{Z}}a_{-s}^{(r+s)} A_{[k-1,r-1]}^{(1)} B_{[\ell-k-1,s-1]}^{(1)}W_{[n-\ell-1,0]}.
\end{align}
For the rest of the argument, it is convenient to perform a change of variable in the quadruple sum in \eqref{eq:weightUnm1}, which is to use $j=r+s$ and $s$ as indices of summation. Then the expression we obtain is  
\begin{align*}
& \sum_{j=2}^{p}\sum_{s=1}^{j-1}\sum_{k\in\mathbb{Z}}\sum_{\ell\in\mathbb{Z}} a_{-s}^{(j)}\,A_{[k-1,j-s-1]}^{(1)} B_{[\ell-k-1,s-1]}^{(1)} W_{[n-\ell-1,0]}\\
= & \sum_{j=2}^{p}\sum_{s=1}^{j-1}\sum_{\ell\in\mathbb{Z}}\sum_{k\in\mathbb{Z}} a_{-s}^{(j)}\, A_{[\ell-1,j-s-1]}^{(1)} B_{[k-\ell-1,s-1]}^{(1)} W_{[n-k-1,0]} 
\end{align*}
where we interchanged the indices $k$ and $\ell$. So we conclude that
\begin{align}\label{eq:weightUnm1bis}
\sum_{\gamma\in\mathcal{U}_{[n,-1]}}w(\gamma) & =\sum_{j=1}^{p}\sum_{k\in\mathbb{Z}}a_{0}^{(j)} A^{(1)}_{[k-1,j-1]} W_{[n-k-1,0]}\notag\\
& +\sum_{j=2}^{p}\sum_{s=1}^{j-1}\sum_{\ell\in\mathbb{Z}}\sum_{k\in\mathbb{Z}} a_{-s}^{(j)}\, A_{[\ell-1,j-s-1]}^{(1)} B_{[k-\ell-1,s-1]}^{(1)} W_{[n-k-1,0]}.
\end{align}
It then follows from \eqref{eq:partPn0}, \eqref{eq:weightUn0}, \eqref{eq:weightpolyUnj}, and \eqref{eq:weightUnm1bis} that for every $n\geq 1$,
\begin{align}\label{eqsummary}
W_{[n,0]} & =a_{0}^{(0)} W_{[n-1,0]}+\sum_{j=1}^{p}\sum_{k\in\mathbb{Z}}a_{-j}^{(j)}\,B_{[k-1,j-1]}^{(1)} W_{[n-k-1,0]}\notag\\
& +\sum_{j=1}^{p}\sum_{k\in\mathbb{Z}}a_{0}^{(j)} A^{(1)}_{[k-1,j-1]} W_{[n-k-1,0]}\notag\\
& +\sum_{j=2}^{p}\sum_{s=1}^{j-1}\sum_{\ell\in\mathbb{Z}}\sum_{k\in\mathbb{Z}}a_{-s}^{(j)}\, A_{[\ell-1,j-s-1]}^{(1)} B_{[k-\ell-1,s-1]}^{(1)} W_{[n-k-1,0]}.
\end{align}
We introduce now the following definition: 
\[
A_{[t,-1]}^{(1)}=B_{[t,-1]}^{(1)}:=\delta_{t,-1}=\begin{cases}
1, & t=-1,\\
0, & t\neq -1.
\end{cases}
\]
We leave to the reader the easy task of verifying that with this notation, \eqref{eqsummary} reduces to the identity
\begin{equation}\label{eq:redWn0}
W_{[n,0]}=a_{0}^{(0)} W_{[n-1,0]}+\sum_{j=1}^{p}\sum_{s=0}^{j}\sum_{\ell\in\mathbb{Z}}\sum_{k\in\mathbb{Z}} a_{-s}^{(j)}\, A_{[\ell-1,j-s-1]}^{(1)} B_{[k-\ell-1,s-1]}^{(1)} W_{[n-k-1,0]},
\end{equation}
valid for every $n\geq 1$. If we define $A_{-1}^{(1)}(z)\equiv B_{-1}^{(1)}(z)=\sum_{k\in\mathbb{Z}}\frac{\delta_{k,-1}}{z^{k+1}}\equiv 1$, then \eqref{eq:redWn0} expresses the relation
\[
z\,W_{0}(z)-1=a_{0}^{(0)}\, W_{0}(z)+\sum_{j=1}^{p}\sum_{s=0}^{j} a_{-s}^{(j)}\, A_{j-s-1}^{(1)}(z)\, B_{s-1}^{(1)}(z)\, W_{0}(z)
\]   
which is equivalent to \eqref{eq:W0expansion}.\qed

\section{Generalized Stieltjes--Rogers polynomials and related families of polynomials}\label{sec:STDyck}

In this section we analyze the collections of lattice paths $\mathcal{R}_{[n,j]}$, $\mathcal{S}_{[n,j]}$, $\widehat{\mathcal{S}}_{[n,j]}$ defined in the introduction. Recall that by definition the paths in these collections have steps of only two types as indicated in \eqref{SRedges}, and we use the notation \eqref{eq:ananp}. So now we have 
\begin{align*}
w((n,m)\rightarrow(n+1,m+1)) & = 1,\\
w((n,m)\rightarrow(n+1,m-p)) & = a_{m-p}.
\end{align*}

The formal series associated with the families of polynomials \eqref{def:wpolyRnj}--\eqref{def:wpolyTnj} are the following expressions, defined for each $0\leq j\leq p$:
\begin{align}
R_{j}(z) & :=\sum_{n=0}^{\infty}\frac{R_{[n,j]}}{z^{n+1}}=\sum_{m=0}^{\infty}\frac{R_{[m(p+1)+j,j]}}{z^{m(p+1)+j+1}},\label{fsRj}\\
S_{j}(z) & :=\sum_{n=0}^{\infty}\frac{S_{[n,j]}}{z^{n+1}}=\sum_{m=0}^{\infty}\frac{S_{[m(p+1)+j,j]}}{z^{m(p+1)+j+1}},\label{fsSj}\\
T_{j}(z) & :=\sum_{n=0}^{\infty}\frac{T_{[n,j]}}{z^{n+1}}=\sum_{m=0}^{\infty}\frac{T_{[m(p+1)+j,j]}}{z^{m(p+1)+j+1}}.\label{fsTj}
\end{align}

We also introduce the families of shifted paths and the corresponding weight polynomials. For an integer $q\geq 0$, let
\begin{align*}
\mathcal{S}^{(q)}_{[n,j]} & :=\{\gamma+q: \gamma\in\mathcal{S}_{[n,j]}\},\\
\widehat{\mathcal{S}}^{(q)}_{[n,j]} & :=\{\gamma-q: \gamma\in\widehat{\mathcal{S}}_{[n,j]}\},
\end{align*}
recall the definition of $\gamma\pm q$ given in the introduction. The associated weight polynomials are
\begin{align}
S^{(q)}_{[n,j]} & :=\sum_{\gamma\in\mathcal{S}_{[n,j]}}w(\gamma+q),\label{def:Sqnj}\\
T^{(q)}_{[n,j]} & :=\sum_{\gamma\in\widehat{\mathcal{S}}_{[n,j]}}w(\gamma-q).\label{def:Tqnj}
\end{align}
The corresponding formal series are the expressions
\begin{align}
S^{(q)}_{j}(z) & :=\sum_{m=0}^{\infty}\frac{S^{(q)}_{[m(p+1)+j,j]}}{z^{m(p+1)+j+1}},\label{fsSjq}\\
T^{(q)}_{j}(z) & :=\sum_{m=0}^{\infty}\frac{T^{(q)}_{[m(p+1)+j,j]}}{z^{m(p+1)+j+1}}.\label{fsTjq}
\end{align}

In the following result we gather some elementary properties of the collections $\mathcal{R}_{[n,j]}$, $\mathcal{S}_{[n,j]}$, $\widehat{\mathcal{S}}_{[n,j]}$, and the corrresponding weight polynomials \eqref{def:wpolyRnj}--\eqref{def:wpolyTnj}.

\begin{proposition}\label{prop:elemprop}
Let $n=m(p+1)+j$, $m\geq 1$, $0\leq j\leq p$. The following properties hold:
\begin{itemize}
\item[$i)$] The initial $p$ steps of any path in $\mathcal{S}_{[n,j]}$ are the upsteps $(k,k)\rightarrow (k+1,k+1)$, $0\leq k\leq p-1$.
\item[$ii)$] If $j=0$, the last step of a path in $\mathcal{S}_{[n,0]}=\mathcal{S}_{[m(p+1),0]}$ is the downstep $(m(p+1)-1,p)\rightarrow(m(p+1),0)$.
\item[$iii)$] $R_{[n,j]}$, $S_{[n,j]}$, and $T_{[n,j]}$ are homogeneous polynomials of degree $m$ in the variables $\{a_{k}: -mp\leq k\leq (m-1)p+j\}$, $\{a_{k}:0\leq k\leq (m-1)p+j\}$, and $\{a_{k}:-j-mp\leq k\leq -p\}$, respectively. 
\item[$iv)$] There is a bijection between the sets $\mathcal{S}_{[n,j]}$ and $\widehat{\mathcal{S}}_{[n,j]}$, established by the map $\gamma\mapsto\widehat{\gamma}$ that is defined as follows: given a path $\gamma\in\mathcal{S}_{[n,j]}$, it is first reflected with respect to the real axis, and the result is then reflected with respect to the vertical line $x=\frac{n}{2}$, to obtain the path $\widehat{\gamma}\in\widehat{\mathcal{S}}_{[n,j]}$. In consequence, if we write $S_{[n,j]}=S_{[n,j]}(a_{0},a_{1},\ldots,a_{(m-1)p+j})$, then $T_{[n,j]}=S_{[n,j]}(a_{-p},a_{-p-1},\ldots,a_{-mp-j})$, i.e., $T_{[n,j]}$ is obtained by replacing in $S_{[n,j]}$ the variable $a_{k}$ by the variable $a_{-p-k}$ for each $0\leq k\leq (m-1)p+j$.
\item[$v)$] We have
\begin{equation}\label{eq:cardinality}
\mathrm{card}(\mathcal{R}_{[n,j]})=\binom{m(p+1)+j}{m},\qquad \mathrm{card}(\mathcal{S}_{[n,j]})=\frac{j+1}{pm+j+1}\binom{m(p+1)+j}{m}.
\end{equation}
\end{itemize}
\end{proposition}
\begin{proof}
Properties $i)$ and $ii)$ are immediately obtained from the definitions of $\mathcal{S}_{[n,j]}$ and $\mathcal{S}_{[m(p+1),0]}$. Property $iii)$ follows from the fact that each path in the collections $\mathcal{R}_{[n,j]}$ and $\widehat{\mathcal{S}}_{[n,j]}$ contains exactly $m$ downsteps (see \eqref{relcountupdown}, which is also valid in $\widehat{\mathcal{S}}_{[n,j]}$). In consequence, in the expressions of the polynomials $R_{[n,j]}$, $S_{[n,j]}$, and $T_{[n,j]}$, each term $w(\gamma)$ is the product of $m$ variables $a_{k}$, counting multiplicities. A simple counting shows that these polynomials are indeed expressed in terms of the variables $a_{k}$ indicated.

It is very easy to see that the map $\gamma\mapsto\widehat{\gamma}$ is well-defined and is a bijection. This map transforms a downstep with weight $a_{k}$, $k\geq 0$, into a downstep with weight $a_{-p-k}$. This implies immediately the indicated relation between the polynomials $S_{[n,j]}$ and $T_{[n,j]}$.

The $m$ downsteps in a path $\gamma\in\mathcal{R}_{[n,j]}$ can be positioned in any order among the $n=m(p+1)+j$ steps in $\gamma$. So it is clear that the cardinality of $\mathcal{R}_{[n,j]}$ is $\binom{m(p+1)+j}{m}$. See Remark~\ref{rem:card} for a justification of the cardinality of $\mathcal{S}_{[n,j]}$ in \eqref{eq:cardinality}.
\end{proof}

Formulas \eqref{eq:cardinality} are also valid if $n=j\in\{0,\ldots,p\}$, in which case we have $\mathrm{card}(\mathcal{R}_{[j,j]})=\mathrm{card}(\mathcal{S}_{[j,j]})=1$. We also have $R_{[j,j]}=S_{[j,j]}=T_{[j,j]}=1$.

\begin{figure}
\begin{center}
\begin{tikzpicture}[scale=0.7]
\draw[line width=1.5pt]  (-3,0) -- (13.5,0);
\draw[line width=1.5pt]  (-3,-2.5) -- (-3,2.5);
\draw[line width=0.7pt] (-3,0) -- (-2,-2) -- (-1,-1) -- (0,0) -- (1,1) -- (2,-1) -- (3,0) -- (4,1) -- (5,2) -- (6,0) -- (7,1) -- (8,-1) -- (9,0) -- (10,-2) -- (11,-1) -- (12,0) -- (13,1);
\draw [line width=0.5] (-2,0) -- (-2,-0.12);
\draw [line width=0.5] (-1,0) -- (-1,-0.12);
\draw [line width=0.5] (0,0) -- (0,-0.12);
\draw [line width=0.5] (1,0) -- (1,-0.12);
\draw [line width=0.5] (2,0) -- (2,-0.12);
\draw [line width=0.5] (3,0) -- (3,-0.12);
\draw [line width=0.5] (4,0) -- (4,-0.12);
\draw [line width=0.5] (5,0) -- (5,-0.12);
\draw [line width=0.5] (6,0) -- (6,-0.12);
\draw [line width=0.5] (7,0) -- (7,-0.12);
\draw [line width=0.5] (8,0) -- (8,-0.12);
\draw [line width=0.5] (9,0) -- (9,-0.12);
\draw [line width=0.5] (10,0) -- (10,-0.12);
\draw [line width=0.5] (11,0) -- (11,-0.12);
\draw [line width=0.5] (12,0) -- (12,-0.12);
\draw [line width=0.5] (13,0) -- (13,-0.12);
\draw [line width=0.5] (-3,1) -- (-3.12,1);
\draw [line width=0.5] (-3,2) -- (-3.12,2);
\draw [line width=0.5] (-3,0) -- (-3.12,0);
\draw [line width=0.5] (-3,-1) -- (-3.12,-1);
\draw [line width=0.5] (-3,-2) -- (-3.12,-2);
\draw [dotted] (-3,1) -- (13.5,1);
\draw [dotted] (-3,2) -- (13.5,2);
\draw [dotted] (-3,-1) -- (13.5,-1);
\draw [dotted] (-3,-2) -- (13.5,-2);
\draw [dotted] (-2,-2.5) -- (-2,2.5);
\draw [dotted] (-1,-2.5) -- (-1,2.5);
\draw [dotted] (0,-2.5) -- (0,2.5);
\draw [dotted] (1,-2.5) -- (1,2.5);
\draw [dotted] (2,-2.5) -- (2,2.5);
\draw [dotted] (3,-2.5) -- (3,2.5);
\draw [dotted] (4,-2.5) -- (4,2.5);
\draw [dotted] (5,-2.5) -- (5,2.5);
\draw [dotted] (6,-2.5) -- (6,2.5);
\draw [dotted] (7,-2.5) -- (7,2.5);
\draw [dotted] (8,-2.5) -- (8,2.5);
\draw [dotted] (9,-2.5) -- (9,2.5);
\draw [dotted] (10,-2.5) -- (10,2.5);
\draw [dotted] (11,-2.5) -- (11,2.5);
\draw [dotted] (12,-2.5) -- (12,2.5);
\draw [dotted] (13,-2.5) -- (13,2.5);
\draw (-2,-0.1) node[below, scale=0.8]{$1$};
\draw (-1,-0.1) node[below, scale=0.8]{$2$};
\draw (0,-0.1) node[below, scale=0.8]{$3$};
\draw (1,-0.1) node[below, scale=0.8]{$4$};
\draw (2,-0.1) node[below, scale=0.8]{$5$};
\draw (3,-0.1) node[below, scale=0.8]{$6$};
\draw (4,-0.1) node[below, scale=0.8]{$7$};
\draw (5,-0.1) node[below, scale=0.8]{$8$};
\draw (6,-0.1) node[below, scale=0.8]{$9$};
\draw (7,-0.1) node[below, scale=0.8]{$10$};
\draw (8,-0.1) node[below, scale=0.8]{$11$};
\draw (9,-0.1) node[below, scale=0.8]{$12$};
\draw (10,-0.1) node[below, scale=0.8]{$13$};
\draw (11,-0.1) node[below, scale=0.8]{$14$};
\draw (12,-0.1) node[below, scale=0.8]{$15$};
\draw (13,-0.1) node[below, scale=0.8]{$16$};
\draw (-3.1,1) node[left, scale=0.8]{$1$};
\draw (-3.1,2) node[left, scale=0.8]{$2$};
\draw (-3.1,0) node[left, scale=0.8]{$0$};
\draw (-3.1,-1) node[left, scale=0.8]{$-1$};
\draw (-3.1,-2) node[left, scale=0.8]{$-2$};
\end{tikzpicture}
\end{center}
\caption{Example, in the case $p=2$, of a path in the collection $\mathcal{R}_{[16,1]}$ with weight $a_{-2}^2\, a_{-1}^2\, a_{0}$.}
\label{genDyckpath}
\end{figure}
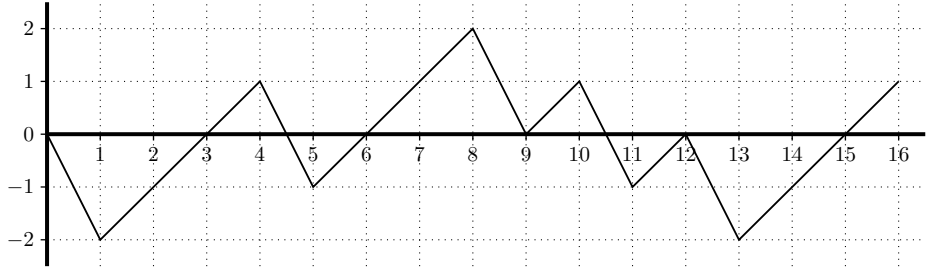

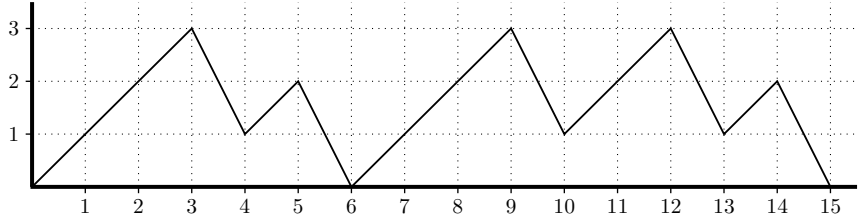
\begin{figure}
\begin{center}
\begin{tikzpicture}[scale=0.7]
\draw[line width=1.5pt]  (-3.03,0) -- (12.5,0);
\draw[line width=1.5pt]  (-3,0) -- (-3,3.5);
\draw[line width=0.7pt] (-3,0) -- (-2,1) -- (-1,2) -- (0,3) -- (1,1) -- (2,2) -- (3,0) -- (4,1) -- (5,2) -- (6,3) -- (7,1) -- (8,2) -- (9,3) -- (10,1) -- (11,2) -- (12,0);
\draw [line width=0.5] (-2,0) -- (-2,-0.12);
\draw [line width=0.5] (-1,0) -- (-1,-0.12);
\draw [line width=0.5] (0,0) -- (0,-0.12);
\draw [line width=0.5] (1,0) -- (1,-0.12);
\draw [line width=0.5] (2,0) -- (2,-0.12);
\draw [line width=0.5] (3,0) -- (3,-0.12);
\draw [line width=0.5] (4,0) -- (4,-0.12);
\draw [line width=0.5] (5,0) -- (5,-0.12);
\draw [line width=0.5] (6,0) -- (6,-0.12);
\draw [line width=0.5] (7,0) -- (7,-0.12);
\draw [line width=0.5] (8,0) -- (8,-0.12);
\draw [line width=0.5] (9,0) -- (9,-0.12);
\draw [line width=0.5] (10,0) -- (10,-0.12);
\draw [line width=0.5] (11,0) -- (11,-0.12);
\draw [line width=0.5] (12,0) -- (12,-0.12);
\draw [line width=0.5] (-3,1) -- (-3.12,1);
\draw [line width=0.5] (-3,2) -- (-3.12,2);
\draw [line width=0.5] (-3,3) -- (-3.12,3);
\draw [dotted] (-3,1) -- (12.5,1);
\draw [dotted] (-3,2) -- (12.5,2);
\draw [dotted] (-3,3) -- (12.5,3);
\draw [dotted] (-2,0) -- (-2,3.5);
\draw [dotted] (-1,0) -- (-1,3.5);
\draw [dotted] (0,0) -- (0,3.5);
\draw [dotted] (1,0) -- (1,3.5);
\draw [dotted] (2,0) -- (2,3.5);
\draw [dotted] (3,0) -- (3,3.5);
\draw [dotted] (4,0) -- (4,3.5);
\draw [dotted] (5,0) -- (5,3.5);
\draw [dotted] (6,0) -- (6,3.5);
\draw [dotted] (7,0) -- (7,3.5);
\draw [dotted] (8,0) -- (8,3.5);
\draw [dotted] (9,0) -- (9,3.5);
\draw [dotted] (10,0) -- (10,3.5);
\draw [dotted] (11,0) -- (11,3.5);
\draw [dotted] (12,0) -- (12,3.5);
\draw (-2,-0.1) node[below, scale=0.8]{$1$};
\draw (-1,-0.1) node[below, scale=0.8]{$2$};
\draw (0,-0.1) node[below, scale=0.8]{$3$};
\draw (1,-0.1) node[below, scale=0.8]{$4$};
\draw (2,-0.1) node[below, scale=0.8]{$5$};
\draw (3,-0.1) node[below, scale=0.8]{$6$};
\draw (4,-0.1) node[below, scale=0.8]{$7$};
\draw (5,-0.1) node[below, scale=0.8]{$8$};
\draw (6,-0.1) node[below, scale=0.8]{$9$};
\draw (7,-0.1) node[below, scale=0.8]{$10$};
\draw (8,-0.1) node[below, scale=0.8]{$11$};
\draw (9,-0.1) node[below, scale=0.8]{$12$};
\draw (10,-0.1) node[below, scale=0.8]{$13$};
\draw (11,-0.1) node[below, scale=0.8]{$14$};
\draw (12,-0.1) node[below, scale=0.8]{$15$};
\draw (-3.1,1) node[left, scale=0.8]{$1$};
\draw (-3.1,2) node[left, scale=0.8]{$2$};
\draw (-3.1,3) node[left, scale=0.8]{$3$};
\end{tikzpicture}
\end{center}
\caption{Example, in the case $p=2$, of a path in the collection $\mathcal{S}_{[15,0]}$ with weight $a_{0}^2\, a_{1}^{3}$.}
\label{Dyckpath}
\end{figure}

\begin{figure}
\begin{center}
\begin{tikzpicture}[scale=0.7]
\draw[line width=1.5pt]  (-3.03,0) -- (12.5,0);
\draw[line width=1.5pt]  (-3,0) -- (-3,-3.5);
\draw[line width=0.7pt] (-3,0) -- (-2,-2) -- (-1,-1) -- (0,-3) -- (1,-2) -- (2,-1) -- (3,-3) -- (4,-2) -- (5,-1) -- (6,0) -- (7,-2) -- (8,-1) -- (9,-3) -- (10,-2) -- (11,-1) -- (12,0);
\draw [line width=0.5] (-2,0) -- (-2,0.12);
\draw [line width=0.5] (-1,0) -- (-1,0.12);
\draw [line width=0.5] (0,0) -- (0,0.12);
\draw [line width=0.5] (1,0) -- (1,0.12);
\draw [line width=0.5] (2,0) -- (2,0.12);
\draw [line width=0.5] (3,0) -- (3,0.12);
\draw [line width=0.5] (4,0) -- (4,0.12);
\draw [line width=0.5] (5,0) -- (5,0.12);
\draw [line width=0.5] (6,0) -- (6,0.12);
\draw [line width=0.5] (7,0) -- (7,0.12);
\draw [line width=0.5] (8,0) -- (8,0.12);
\draw [line width=0.5] (9,0) -- (9,0.12);
\draw [line width=0.5] (10,0) -- (10,0.12);
\draw [line width=0.5] (11,0) -- (11,0.12);
\draw [line width=0.5] (12,0) -- (12,0.12);
\draw [line width=0.5] (-3,-1) -- (-3.12,-1);
\draw [line width=0.5] (-3,-2) -- (-3.12,-2);
\draw [line width=0.5] (-3,-3) -- (-3.12,-3);
\draw [dotted] (-3,-1) -- (12.5,-1);
\draw [dotted] (-3,-2) -- (12.5,-2);
\draw [dotted] (-3,-3) -- (12.5,-3);
\draw [dotted] (-2,0) -- (-2,-3.5);
\draw [dotted] (-1,0) -- (-1,-3.5);
\draw [dotted] (0,0) -- (0,-3.5);
\draw [dotted] (1,0) -- (1,-3.5);
\draw [dotted] (2,0) -- (2,-3.5);
\draw [dotted] (3,0) -- (3,-3.5);
\draw [dotted] (4,0) -- (4,-3.5);
\draw [dotted] (5,0) -- (5,-3.5);
\draw [dotted] (6,0) -- (6,-3.5);
\draw [dotted] (7,0) -- (7,-3.5);
\draw [dotted] (8,0) -- (8,-3.5);
\draw [dotted] (9,0) -- (9,-3.5);
\draw [dotted] (10,0) -- (10,-3.5);
\draw [dotted] (11,0) -- (11,-3.5);
\draw [dotted] (12,0) -- (12,-3.5);
\draw (-2,0.1) node[above, scale=0.8]{$1$};
\draw (-1,0.1) node[above, scale=0.8]{$2$};
\draw (0,0.1) node[above, scale=0.8]{$3$};
\draw (1,0.1) node[above, scale=0.8]{$4$};
\draw (2,0.1) node[above, scale=0.8]{$5$};
\draw (3,0.1) node[above, scale=0.8]{$6$};
\draw (4,0.1) node[above, scale=0.8]{$7$};
\draw (5,0.1) node[above, scale=0.8]{$8$};
\draw (6,0.1) node[above, scale=0.8]{$9$};
\draw (7,0.1) node[above, scale=0.8]{$10$};
\draw (8,0.1) node[above, scale=0.8]{$11$};
\draw (9,0.1) node[above, scale=0.8]{$12$};
\draw (10,0.1) node[above, scale=0.8]{$13$};
\draw (11,0.1) node[above, scale=0.8]{$14$};
\draw (12,0.1) node[above, scale=0.8]{$15$};
\draw (-3.1,-1) node[left, scale=0.8]{$-1$};
\draw (-3.1,-2) node[left, scale=0.8]{$-2$};
\draw (-3.1,-3) node[left, scale=0.8]{$-3$};
\end{tikzpicture}
\end{center}
\caption{Example, in the case $p=2$, of a path in the collection $\widehat{\mathcal{S}}_{[15,0]}$ with weight $a_{-2}^{2}\,a_{-3}^{3}$. If $\gamma$ denotes the path in Figure~\ref{Dyckpath}, its reflection $\widehat{\gamma}$ is shown here, constructed as indicated in Proposition~\ref{prop:elemprop}.$iv)$.}
\label{reflectedDyckpath}
\end{figure}
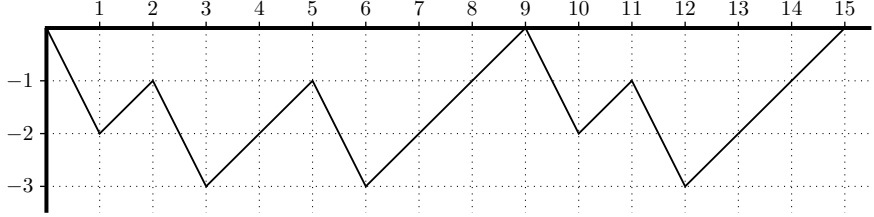

\bigskip

\noindent\textbf{Proof of Theorem~\ref{theo:weightpoly}:} Let us define the collection of multi-indices
\[
\mathcal{I}_{[m,j]}:=\{(i_{1},\ldots,i_{m})\in\mathbb{Z}_{\geq 0}^{m} : 0\leq i_{1}\leq j,\,\,\mbox{and}\,\,0\leq i_{k}\leq i_{k-1}+p\,\,\mbox{for each}\,\,k=2,\ldots,m\}.
\] 
For a path $\gamma\in\mathcal{S}_{[m(p+1)+j,j]}$, let $d_{k}=d_{k}(\gamma)$, $k=1,\ldots,m$, denote the $k$th downstep in $\gamma$, \emph{counting from right to left} (recall that a path in $\mathcal{S}_{[m(p+1)+j,j]}$ has $m$ downsteps). We consider the map $\eta$ defined on $\mathcal{S}_{[m(p+1)+j,j]}$ by $\eta(\gamma)=(i_{1},\ldots,i_{m})\in\mathbb{Z}_{\geq 0}^{m}$, where $a_{i_{k}}=w(d_{k})$, $k=1,\ldots,m,$ is the weight of  
the downstep $d_{k}$ in $\gamma$ (in this proof we see the $a_{n}$'s as variables rather than complex numbers). 

Let us show that $\eta(\gamma)\in\mathcal{I}_{[m,j]}$ for every $\gamma\in\mathcal{S}_{[m(p+1)+j,j]}$, and that $\eta: \mathcal{S}_{[m(p+1)+j,j]}\longrightarrow \mathcal{I}_{[m,j]}$ is a bijection. This will immediately imply \eqref{eq:descSnj}. 

First, recall that a path $\gamma\in\mathcal{S}_{[m(p+1)+j,j]}$ ends at the point $(m(p+1)+j,j)$. Therefore, the last downstep $d_{1}$ in $\gamma$ is clearly one with weight $w(d_{1})=a_{i_{1}}$, $0\leq i_{1}\leq j$. For a fixed $1\leq k\leq m-1$, if the downstep $d_{k}$ in $\gamma$ has weight $a_{i_{k}}$, $i_{k}\geq 0$, then by the geometry of the path the following downstep $d_{k+1}$ to the left of $d_{k}$ has necessarily weight $a_{i_{k+1}}$ with $0\leq i_{k+1}\leq i_{k}+p$. This justifies that $\eta(\gamma)\in\mathcal{I}_{[m,j]}$.

Assume that $\eta(\gamma_{1})=\eta(\gamma_{2})$, let $\eta(\gamma_{r})=(i_{1}^{(r)},\ldots,i_{m}^{(r)})$, and let $d_{k}^{(r)}$ be the $k$th downstep in $\gamma_{r}$ as defined above, $r=1, 2$. The equality $i_{1}^{(1)}=i_{1}^{(2)}$ implies that the last downsteps $d_{1}^{(1)}$ and $d_{1}^{(2)}$ occupy the same positions in $\gamma_{1}$ and $\gamma_{2}$, respectively. From $i_{2}^{(1)}=i_{2}^{(2)}$ it then follows that the downsteps $d_{2}^{(1)}$ and $d_{2}^{(2)}$ occupy the same positions, and so on. The equality $\gamma_{1}=\gamma_{2}$ follows easily from an induction argument.

To justify that $\eta$ is onto, consider $(i_{1},\ldots,i_{m})\in\mathcal{I}_{[m,j]}$. Let $\gamma$ be the path constructed inductively as follows, \emph{from right to left}. Let the last $j-i_{1}$ steps in $\gamma$ be consecutive upsteps connecting the points $(m(p+1)+i_{1},i_{1})$ and $(m(p+1)+j,j)$, and let the last downstep $d_{1}$ in $\gamma$ be the downstep with terminal height $i_{1}$ preceding these $j-i_{1}$ upsteps. Assume we have defined the steps in $\gamma$ from right to left up to the position of the downstep $d_{\ell}$, $1\leq \ell\leq m-1$, in such a way that $w(d_{k})=a_{i_{k}}$ for each $1\leq k\leq \ell$, and the number of upsteps in $\gamma$ located after the downstep $d_{\ell}$ is 
\[
(j-i_{1})+(i_{1}+p-i_{2})+(i_{2}+p-i_{3})+\cdots+(i_{\ell-1}+p-i_{\ell})=j+(\ell-1)p-i_{\ell}.
\]
Note that each term in parenthesis is non-negative. Then let $d_{\ell+1}$ be the downstep in $\gamma$ that precedes $d_{\ell}$ with $i_{\ell}+p-i_{\ell+1}$ intermediate upsteps between $d_{\ell}$ and $d_{\ell+1}$. This is possible since $i_{\ell+1}\leq i_{\ell}+p$. Then $w(d_{\ell+1})=a_{i_{\ell+1}}$, and the number of upsteps in $\gamma$ located after $d_{\ell+1}$ is $(j+(\ell-1)p-i_{\ell})+(i_{\ell}+p-i_{\ell+1})=j+\ell p-i_{\ell+1}$. We can continue this process and define $\gamma$ up to the position of the first downstep $d_{m}$, so that the number of upsteps located after $d_{m}$ is $j+(m-1)p-i_{m}$, and $d_{m}$ has terminal height $i_{m}$. Finally, we complete the construction of $\gamma$ defining the first $i_{m}+p$ steps of $\gamma$ that connect $(0,0)$ with $(i_{m}+p,i_{m}+p)$ to be upsteps. This is possible since $(i_{m}+p)+(j+(m-1)p-i_{m})=mp+j$. Then, $\gamma$ is obviously a path in $\mathcal{S}_{[m(p+1)+j,j]}$ and $\eta(\gamma)=(i_{1},\ldots,i_{m})$. This concludes the proof that $\eta$ is a bijection. 

The proof of \eqref{eq:descRnj} goes along the same lines. In this case we consider the collection
\[
\widetilde{\mathcal{I}}_{[m,j]}:=\{(i_{1},\ldots,i_{m})\in\mathbb{Z}^{m}:-p\leq i_{1}\leq (m-1)p+j\,\,\mbox{and}\,\,i_{k-1}-p\leq i_{k}\leq (m-k)p+j,\,\,2\leq k\leq m\}.
\]
Given a path $\gamma\in\mathcal{R}_{[m(p+1)+j,j]}$, let $\widetilde{d}_{k}=\widetilde{d}_{k}(\gamma)$ denote the $k$th downstep in $\gamma$, now \emph{counting from left to right}, and let $\widetilde{\eta}$ be the map defined on $\mathcal{R}_{[m(p+1)+j,j]}$ by $\widetilde{\eta}(\gamma)=(i_{1},\ldots,i_{m})$, where $a_{i_{k}}=w(\widetilde{d}_{k})$ for each $k=1,\ldots,m$. We leave to the reader the task of checking that $\widetilde{\eta}:\mathcal{R}_{[m(p+1)+j,j]}\longrightarrow\widetilde{\mathcal{I}}_{[m,j]}$ is a bijection, which immediately implies \eqref{eq:descRnj}.

The identity \eqref{eq:descTnj} follows immediately from \eqref{eq:descSnj} and Proposition~\ref{prop:elemprop}.$iv)$.\qed

\bigskip

From \eqref{eq:descSnj} and \eqref{def:Sqnj}--\eqref{def:Tqnj} we also obtain that for any integer $q\geq 0$,
\begin{align}
S_{[m(p+1)+j,j]}^{(q)} & =\sum_{i_{1}=q}^{j+q}\,\,\sum_{i_{2}=q}^{i_{1}+p}\,\,\sum_{i_{3}=q}^{i_{2}+p}\,\,\cdots\,\,\sum_{i_{m}=q}^{i_{m-1}+p}\,\,\prod_{k=1}^{m}a_{i_{k}},\label{eq:descSnjq}\\
T_{[m(p+1)+j,j]}^{(q)} & =\sum_{i_{1}=-j-p-q}^{-p-q}\,\,\sum_{i_{2}=i_{1}-p}^{-p-q}\,\,\sum_{i_{3}=i_{2}-p}^{-p-q}\,\,\cdots\,\,\sum_{i_{m}=i_{m-1}-p}^{-p-q}\,\,\prod_{k=1}^{m}a_{i_{k}}.\notag
\end{align}

The following result is due to Aptekarev, Kaliaguine, and Van Iseghem \cite[Corollary 1]{AptKalVan}, who proved it using exclusively the identities \eqref{eq:descSnj} and \eqref{eq:descSnjq}. We present the result as an immediate consequence of the lattice path representations of the polynomials $S_{[n,j]}$ and $S_{[n,j]}^{(q)}$, and of Theorem~\ref{theo:relseriesAs}. 

\begin{corollary}[Aptekarev, Kaliaguine, Van Iseghem \cite{AptKalVan}]\label{cor:Aseries}
The following relations hold between the series defined in \eqref{fsSj} and \eqref{fsSjq}:
\begin{align}
S_{0}(z) & =\frac{1}{z-a_{0}\,S_{p-1}^{(1)}(z)}\label{eq:S0Sp}\\
S_{j}(z) & =S_{0}(z)\,S^{(1)}_{j-1}(z)\qquad 1\leq j\leq p.\label{sr:eq:SjS0Sjm}
\end{align}
We also have
\begin{align}
z S_{0}(z)-1 & =a_{0}\,S_{p}(z)\label{eq:S0Si:1}\\
S_{j}(z) & =S_{i}(z)\,S_{j-i-1}^{(i+1)}(z)\qquad 0\leq i<j\leq p.\label{sr:eq:S0Si:2}
\end{align}
\end{corollary}
\begin{proof}
The collections $\mathcal{S}_{[n,j]}$ and $\widehat{\mathcal{S}}_{[n,j]}^{(q)}$ are subcollections of the collections $\mathcal{D}_{[n,j]}$ and $\widehat{\mathcal{D}}_{[n,j]}^{(q)}$. The weights of the paths in $\mathcal{D}_{[n,j]}$ and $\widehat{\mathcal{D}}_{[n,j]}^{(q)}$ are expressed in terms of the given sequences of complex numbers $(a_{n}^{(k)})_{n\in\mathbb{Z}}$, $0\leq k\leq p$. If we now take
\begin{equation}\label{eq:redank}
\begin{aligned}
a_{n}^{(k)} & =0,\qquad 0\leq k\leq p-1,\quad n\in\mathbb{Z},\\
a_{n}^{(p)} & =a_{n},\qquad n\in\mathbb{Z},
\end{aligned}
\end{equation}
(cf. \eqref{eq:ananp}) then the polynomials $A_{[n,j]}$ and $A_{[n,j]}^{(q)}$ reduce to the polynomials $S_{[n,j]}$ and $S_{[n,j]}^{(q)}$, respectively. Therefore,
\[
A_{j}(z)=S_{j}(z),\qquad A_{j}^{(q)}(z)=S_{j}^{(q)}(z),\qquad \mbox{provided \eqref{eq:redank} holds}.
\]
So the relations \eqref{eq:S0Sp}--\eqref{sr:eq:S0Si:2} follow from \eqref{eq:A0Aks}--\eqref{eq:A0Ai:2}.
\end{proof}

\begin{remark}\label{rem:card}
Let $m\geq 0$, $0\leq j\leq p$. With the help of the relations in Corollary~\ref{cor:Aseries}, we can justify the identity
\begin{equation}\label{eq:proofcardDnj}
\mathrm{card}(\mathcal{S}_{[m(p+1)+j,j]})=\frac{j+1}{pm+j+1}\binom{m(p+1)+j}{m}.
\end{equation}
Indeed, let $C_{[m,j]}:=\mathrm{card}(\mathcal{S}_{[m(p+1)+j,j]})$. If we take in \eqref{eq:ananp} the weights $a_{k}=1$ for all $k\geq 0$, then each path in the collections $\mathcal{S}_{[m(p+1)+j,j]}$ and $\mathcal{S}_{[m(p+1)+j,j]}^{(q)}$, $q\geq 1$, has weight $1$, hence in this case we have $S_{[m(p+1)+j,j]}=S_{[m(p+1)+j,j]}^{(q)}=C_{[m,j]}$. Therefore,
\begin{equation}\label{eq:newformAj}
S_{j}(z)=S_{j}^{(q)}(z)=\sum_{m=0}^{\infty}\frac{C_{[m,j]}}{z^{m(p+1)+j+1}},\qquad \mbox{for all}\,\,\,\,0\leq j\leq p,\,\,\,q\geq 1.
\end{equation}
This implies, in virtue of \eqref{eq:S0Sp} and \eqref{sr:eq:SjS0Sjm}, the relations
\begin{align}
S_{0}(z) & =\frac{1}{z-S_{0}^{p}(z)},\label{eq:relA0A0id}\\
S_{j}(z) & =S_{0}(z)\,S_{j-1}(z),\qquad 1\leq j\leq p,\label{eq:relAjA0id}
\end{align}
and from \eqref{eq:relAjA0id} we obtain
\begin{equation}\label{eq:relSjS0jp1}
S_{j}(z)=S_{0}(z)^{j+1},\qquad 1\leq j\leq p.
\end{equation}
Consider the function
\[
h(w):=S_{0}\left(\frac{1}{w}\right)=\sum_{m=0}^{\infty}C_{[m,0]}w^{m(p+1)+1}.
\]
Using the estimate $C_{[m,0]}=\mathrm{card}(\mathcal{S}_{[m(p+1),0]})\leq \mathrm{card}(\mathcal{R}_{[m(p+1),0]})=\binom{m(p+1)}{m}$, we observe that the power series defining $h$ has positive radius of convergence, so $h$ is analytic in a neighborhood of the origin. The relation \eqref{eq:relA0A0id} implies that $h$ satisfies
\[
w=\frac{h(w)}{h(w)^{p+1}+1}.
\]
Therefore $h$ is the inverse of the function $f(t)=\frac{t}{\phi(t)}$, where $\phi(t)=t^{p+1}+1$.  

Given a power series $g(z)=\sum_{k=0}^{\infty} c_{k} z^{k}$, we use now the traditional notation $c_{k}=[z^k]\,g(z)$. The Lagrange inversion formula (cf. \cite{Gessel}) applied to $h$ and $f$ asserts that
\[
[w^{n}]\,h(w)=\frac{1}{n}\,[t^{n-1}]\,\phi(t)^{n},\qquad n\geq 1.
\]
Applying this identity with $n=m(p+1)+1$, we get
\begin{align*}
C_{[m,0]} & =[w^{m(p+1)+1}]\,h(w)\\
& =\frac{1}{m(p+1)+1}\,[t^{m(p+1)}]\,(t^{p+1}+1)^{m(p+1)+1}\\
 & =\frac{1}{m(p+1)+1}\binom{m(p+1)+1}{m}\\ & =\frac{1}{mp+1}\binom{m(p+1)}{m}.
\end{align*}
This justifies \eqref{eq:proofcardDnj} for $j=0$. 

A more general version of the Lagrange inversion formula is
\begin{equation}\label{eq:Laginv}
[w^{n}]\,h(w)^{k}=\frac{k}{n}\,[t^{n-k}]\,\phi(t)^{n},\qquad n,k\geq 1.
\end{equation}
According to \eqref{eq:relSjS0jp1}, we have 
\[
S_{j}\left(\frac{1}{w}\right)=\sum_{m=0}^{\infty} C_{[m,j]} w^{m(p+1)+j+1}=h(w)^{j+1},
\] 
so applying \eqref{eq:Laginv} with $n=m(p+1)+j+1$ and $k=j+1$, we obtain
\begin{align*}
C_{[m,j]} & =[w^{m(p+1)+j+1}]\,h(w)^{j+1}\\
& =\frac{j+1}{m(p+1)+j+1}\,[t^{m(p+1)}]\,(t^{p+1}+1)^{m(p+1)+j+1}\\
 & =\frac{j+1}{m(p+1)+j+1}\binom{m(p+1)+j+1}{m}\\ & =\frac{j+1}{mp+j+1}\binom{m(p+1)+j}{m}
\end{align*}
so \eqref{eq:proofcardDnj} is justified.

The numbers in \eqref{eq:proofcardDnj} are known as the \emph{Fuss-Catalan numbers}. 
\end{remark}

\begin{corollary}
The following relations hold between the series defined in \eqref{fsRj}, \eqref{fsSjq}, and \eqref{fsTjq}:
\begin{align}
R_{0}(z) &=\frac{1}{z-\sum_{\ell=0}^{p}a_{-\ell}\,S_{p-\ell-1}^{(1)}(z)\,T_{\ell-1}^{(1)}(z)} \label{sr:eq:W0expansion}\\
R_{j}(z) & =R_{0}(z)\,S_{j-1}^{(1)}(z)\qquad 1\leq j\leq p.\label{sr:eq:WjW0Aj}
\end{align}
In \eqref{sr:eq:W0expansion}, we understand that $S^{(1)}_{-1}(z)\equiv T_{-1}^{(1)}(z)\equiv 1$. More generally,
\begin{equation}\label{sr:eq:WjWiAj}
R_{j}(z)=R_{i}(z)\,S_{j-i-1}^{(i+1)}(z) \qquad 0\leq i<j\leq p.
\end{equation}
\end{corollary}
\begin{proof}
As in the proof of Corollary~\ref{cor:Aseries}, if we define the sequences $(a_{n}^{(k)})_{n\in\mathbb{Z}}$, $0\leq k\leq p$, so that \eqref{eq:redank} holds, then for every $0\leq j\leq p$ and $q\geq 1$ we have the identifications
\[
W_{j}(z)=R_{j}(z),\qquad A_{j}^{(q)}(z)=S_{j}^{(q)}(z),\qquad B_{j}^{(q)}(z)=T_{j}^{(q)}(z),
\]
so the relations \eqref{sr:eq:W0expansion}--\eqref{sr:eq:WjWiAj} follow from \eqref{eq:W0expansion}--\eqref{eq:WjWiAj}.
\end{proof}

\section{Resolvent functions of banded Hessenberg operators and vector continued fractions}\label{sec:resolvcf}

In this section we show that the formal power series defined in \eqref{eq:def:seriesAj}--\eqref{eq:def:seriesBjq} can be identified as resolvent functions of certain banded Hessenberg operators. We also discuss the vector continued fraction expansion for the vector $(A_{0}(z),\ldots,A_{p-1}(z))$.

Again let $(a_{n}^{(k)})_{n\in\mathbb{Z}}$, $0\leq k\leq p$, be a collection of $p+1$ bi-infinite sequences of complex numbers. Let $\{e_{j}\}_{j=0}^{\infty}$ denote the standard basis vectors in the space $\ell^{2}(\mathbb{Z}_{\geq 0})$, where the inner product of two vectors $x=(x_{n})_{n=0}^{\infty}$ and $y=(y_{n})_{n=0}^{\infty}$ will be denoted $\langle x,y \rangle=\sum_{n=0}^{\infty}x_{n} \overline{y_{n}}$. Let $D_{0}$ denote the subspace consisting of all finite linear combinations of the basis vectors $e_{j}$. For an integer $q\geq 0$, let $\mathcal{H}_{q}$ be the operator (possibly unbounded) defined by 
\begin{equation}\label{def:opHcal}
\begin{cases}
\mathcal{H}_{q} e_{0}=\sum_{k=0}^{p} a_{q}^{(k)} e_{k},\\
\mathcal{H}_{q} e_{n}=e_{n-1}+\sum_{k=0}^{p} a_{n+q}^{(k)}\,e_{n+k},\quad n\geq 1,
\end{cases}
\end{equation}
and extended by linearity to $D_{0}$. Likewise, let $\mathcal{B}_{q}$ be the linear operator on $D_{0}$ defined by
\begin{equation}\label{def:opBcal}
\begin{cases}
\mathcal{B}_{q} e_{0} =\sum_{k=0}^{p} a_{-k-q}^{(k)}\, e_{k},\\
\mathcal{B}_{q} e_{n}=e_{n-1}+\sum_{k=0}^{p} a_{-k-q-n}^{(k)}\, e_{n+k},\quad n\geq 1.
\end{cases}
\end{equation}
In the basis $\{e_{j}\}_{j=0}^{\infty}$, the matrix representations of the operators $\mathcal{H}_{q}$ and $\mathcal{B}_{q}$ are the banded lower Hessenberg matrices
\begin{equation}\label{matrixrep}
\begin{pmatrix}
a_{q}^{(0)} & 1 & & & \\
\vdots & a_{q+1}^{(0)} & 1 & \\
a_{q}^{(p)} & \vdots & a_{q+2}^{(0)} & \ddots \\
 & a_{q+1}^{(p)} & \vdots & \ddots \\
 &  & a_{q+2}^{(p)} & \\
 & & & \ddots \\
 &   
\end{pmatrix}
\hspace{2cm}
\begingroup\BigColSep\begin{pmatrix}
a_{-q}^{(0)} & 1 & & & & & & & \\
\vdots & a_{-q-1}^{(0)} & 1 & & \\
a_{-p-q}^{(p)} & \vdots & a_{-q-2}^{(0)} & \ddots & \\
 & a_{-p-q-1}^{(p)} & \vdots & \ddots & \\
 &  & a_{-p-q-2}^{(p)} & & \\
 & & & \ddots & \\
 &   
\end{pmatrix}\endgroup.
\end{equation}
We introduce the resolvent functions
\begin{align}
\phi_{j}^{(q)}(z) & :=\langle(zI-\mathcal{H}_{q})^{-1} e_{j},e_{0}\rangle=\sum_{n=0}^{\infty}\frac{\langle\mathcal{H}_{q}^{n}\, e_{j},e_{0}\rangle}{z^{n+1}},\qquad 0\leq j\leq p,\quad q\geq 0,\label{def:resolphij}\\
\beta_{j}^{(q)}(z) & :=\langle(zI-\mathcal{B}_{q})^{-1} e_{j},e_{0}\rangle =\sum_{n=0}^{\infty}\frac{\langle\mathcal{B}_{q}^{n} \,e_{j},e_{0}\rangle}{z^{n+1}},\qquad 0\leq j\leq p,\quad q\geq 0,\label{def:resolbetaj}
\end{align}
understood as formal power series, which are clearly well-defined. In the case $q=0$ we write $\phi_{j}=\phi_{j}^{(0)}$, $\beta_{j}=\beta_{j}^{(0)}$. We also set $\phi_{-1}^{(q)}(z)\equiv \beta_{-1}^{(q)}(z)\equiv 1$.

Now let $\{\hat{e}_{n}\}_{n\in\mathbb{Z}}$ denote the standard basis in the space $\ell^{2}(\mathbb{Z})$, and let $D_{1}$ be the subspace formed by all finite linear combinations of the basis vectors $\hat{e}_{j}$. Let $\mathcal{W}$ be the linear operator on $D_{1}$ that acts on the standard basis vectors as follows:
\begin{equation}\label{def:operatorW}
\mathcal{W}\,\hat{e}_{n}=\hat{e}_{n-1}+\sum_{k=0}^{p}a_{n}^{(k)} \hat{e}_{n+k},\quad n\in\mathbb{Z},
\end{equation}
and let
\begin{equation}\label{def:resolpsij}
\psi_{j}(z):=\langle(zI-\mathcal{W})^{-1}\hat{e}_{j},\hat{e}_{0}\rangle=\sum_{n=0}^{\infty}\frac{\langle\mathcal{W}^{n}\,\hat{e}_{j},\hat{e}_{0}\rangle}{z^{n+1}},\qquad 0\leq j\leq p,
\end{equation}
understood again as a formal power series.

\begin{proposition}\label{prop:eqresollat}
For each $0\leq j\leq p$ and $q\geq 0$, we have the following identities:
\begin{align}
\phi_{j}^{(q)}(z) & =A_{j}^{(q)}(z),\label{eq:idphiAj}\\
\beta_{j}^{(q)}(z) & =B_{j}^{(q)}(z),\label{eq:idbetaBj}\\
\psi_{j}(z) & = W_{j}(z).\label{eq:idpsiWj}
\end{align}
\end{proposition}
\begin{proof}
First we show that for any integers $n\geq 0$ and $r\geq 0$, we have
\begin{equation}\label{reloppath1}
\langle\mathcal{H}_{q}^{n}\,e_{r}, e_{0}\rangle=\sum_{\gamma\in\mathcal{S}_{1}(n,r)} w(\gamma)
\end{equation}
where $\mathcal{S}_{1}(n,r)$ is the collection of all lattice paths $\gamma$ of length $n$, with initial point $(0,q)$, terminal point $(n,r+q)$, that satisfy $\min(\gamma)=q$. We prove \eqref{reloppath1} by induction on $n$. It is clear that \eqref{reloppath1} holds for $n=0$ since in this case
\[
\langle\mathcal{H}_{q}^{n}\,e_{r}, e_{0}\rangle=\langle e_{r},e_{0}\rangle=\delta_{r,0}=\sum_{\gamma\in\mathcal{S}_{1}(0,r)}w(\gamma).
\]
Recall that the weight of a path of length zero is by definition $1$. Assume that \eqref{reloppath1} is valid for a particular $n\geq 0$ and for all $r\geq 0$. Then from \eqref{def:opHcal} we obtain
\[
\langle\mathcal{H}_{q}^{n+1} e_{r},e_{0}\rangle=\langle\mathcal{H}_{q}^{n}(\mathcal{H}_{q}\,e_{r}),e_{0}\rangle=\begin{cases}
\sum_{k=0}^{p} a_{q}^{(k)}\langle\mathcal{H}_{q}^{n} e_{k},e_{0}\rangle, & \mbox{if}\,\,r=0,\\[0.5em]
\langle\mathcal{H}_{q}^{n} e_{r-1},e_{0}\rangle+\sum_{k=0}^{p} a_{r+q}^{(k)}\langle \mathcal{H}_{q}^{n} e_{r+k},e_{0}\rangle, & \mbox{if}\,\,r\geq 1.
\end{cases}
\]
Suppose that $r=0$. It is evident that the last step of a path in $\mathcal{S}_{1}(n+1,0)$ is one of the steps
\[
(n,k+q)\rightarrow(n+1,q),\qquad 0\leq k\leq p,
\]
with weights $a_{q}^{(k)}$. Therefore, it is clear that 
\[
\sum_{\gamma\in\mathcal{S}_{1}(n+1,0)}w(\gamma)=\sum_{k=0}^{p}a_{q}^{(k)}\Big(\sum_{\gamma\in\mathcal{S}_{1}(n,k)} w(\gamma)\Big)=\sum_{k=0}^{p} a_{q}^{(k)}\langle\mathcal{H}_{q}^{n} e_{k},e_{0}\rangle=\langle\mathcal{H}_{q}^{n+1} e_{0},e_{0}\rangle
\]
where in the second equality we used by induction hypothesis.

Similarly, if $r\geq 1$, the last step of a path in $\mathcal{S}_{1}(n+1,r)$ is one of the steps
\[
(n,r+q+k)\rightarrow(n+1,r+q),\qquad -1\leq k\leq p,
\]
with weight $1$ if $k=-1$, and weights $a_{r+q}^{(k)}$ if $0\leq k\leq p$. Therefore
\begin{align*}
\sum_{\gamma\in\mathcal{S}_{1}(n+1,r)}w(\gamma) & =\sum_{\gamma\in\mathcal{S}_{1}(n,r-1)}w(\gamma)+\sum_{k=0}^{p}a_{r+q}^{(k)}\Big(\sum_{\gamma\in\mathcal{S}_{1}(n,r+k)} w(\gamma)\Big)\\
& =\langle\mathcal{H}_{q}^{n} e_{r-1},e_{0}\rangle+\sum_{k=0}^{p}a_{r+q}^{(k)}\langle\mathcal{H}_{q}^{n} e_{r+k},e_{0}\rangle\\
& =\langle\mathcal{H}_{q}^{n+1} e_{r},e_{0}\rangle
\end{align*}
where in the second equality we used the induction hypothesis. This concludes the proof of \eqref{reloppath1}. It is clear that $\mathcal{S}_{1}(n,j)=\mathcal{D}_{[n,j]}^{(q)}$ for $0\leq j\leq p$, hence $\langle\mathcal{H}_{q}^{n}\, e_{j},e_{0}\rangle=A_{[n,j]}^{(q)}$ for all $n\geq 0$ and \eqref{eq:idphiAj} follows.

For the justification of \eqref{eq:idbetaBj}, we show that for any integers $n\geq 0$ and $r\geq 0$, we have
\begin{equation}\label{reloppath2}
\langle\mathcal{B}_{q}^{n}\, e_{r}, e_{0}\rangle=\sum_{\gamma\in\mathcal{S}_{2}(n,r)} w(\gamma),
\end{equation}
where $\mathcal{S}_{2}(n,r)$ is the collection of all paths $\gamma$ of length $n$, with initial point $(0,-r-q)$, terminal point $(n,-q)$, that satisfy $\max(\gamma)=-q$. We prove \eqref{reloppath2} by induction on $n$. 

It is obvious that \eqref{reloppath2} holds for $n=0$. Assume that it holds for a particular $n\geq 0$ and for all $r\geq 0$. Then from \eqref{def:opBcal} we obtain
\[
\langle\mathcal{B}_{q}^{n+1} e_{r},e_{0}\rangle=\langle\mathcal{B}_{q}^{n}(\mathcal{B}_{q}\,e_{r}),e_{0}\rangle=\begin{cases}
\sum_{k=0}^{p} a_{-k-q}^{(k)}\langle\mathcal{B}_{q}^{n} e_{k},e_{0}\rangle, & \mbox{if}\,\,r=0,\\[0.5em]
\langle\mathcal{B}_{q}^{n} e_{r-1},e_{0}\rangle+\sum_{k=0}^{p} a_{-k-q-r}^{(k)}\langle \mathcal{B}_{q}^{n} e_{r+k},e_{0}\rangle, & \mbox{if}\,\,r\geq 1.
\end{cases}
\]
In the case $r=0$, the collection $\mathcal{S}_{2}(n+1,r)=\mathcal{S}_{2}(n+1,0)$ consists of all paths with initial point $(0,-q)$, terminal point $(n+1,-q)$, and satisfying $\max(\gamma)=-q$. So the first step of a path $\gamma\in\mathcal{S}_{2}(n+1,0)$ is one of the steps
\[
(0,-q)\rightarrow(1,-k-q),\qquad 0\leq k\leq p,
\]
with weights $a_{-k-q}^{(k)}$. The remaining portion of $\gamma$, if we shift it horizontally one unit to the left, can be identified with a path in $\mathcal{S}_{2}(n,k)$. This horizontal shift leaves the weight of a path invariant. Hence we conclude that
\[
\sum_{\gamma\in\mathcal{S}_{2}(n+1,0)}w(\gamma)=\sum_{k=0}^{p}a_{-k-q}^{(k)}\Big(\sum_{\gamma\in\mathcal{S}_{2}(n,k)} w(\gamma)\Big)=\sum_{k=0}^{p}a_{-k-q}^{(k)}\langle\mathcal{B}_{q}^{n}e_{k},e_{0}\rangle=\langle\mathcal{B}_{q}^{n+1}e_{0},e_{0}\rangle
\]
where we used the induction hypothesis in the second equality.

Now suppose that $r\geq 1$. If $\gamma$ is a path in $\mathcal{S}_{2}(n+1,r)$, then the first step of $\gamma$ is one of the steps
\[
(0,-r-q)\rightarrow(1,-k-q-r),\qquad -1\leq k\leq p,
\]
with weight $1$ if $k=-1$, and weights $a_{-k-q-r}^{(k)}$ if $0\leq k\leq p$. The remaining part of $\gamma$ is identifiable with a path in $\mathcal{S}_{2}(n,r+k)$ (after a horizontal shift one unit to the left). So we deduce that
\begin{align*}
\sum_{\gamma\in\mathcal{S}_{2}(n+1,r)}w(\gamma) & =\sum_{\gamma\in\mathcal{S}_{2}(n,r-1)}w(\gamma)+\sum_{k=0}^{p}a_{-k-q-r}^{(k)}\Big(\sum_{\gamma\in\mathcal{S}_{2}(n,r+k)} w(\gamma)\Big)\\
& =\langle\mathcal{B}_{q}^{n} e_{r-1},e_{0}\rangle+\sum_{k=0}^{p}a_{-k-q-r}^{(k)}\langle\mathcal{B}_{q}^{n} e_{r+k},e_{0}\rangle\\
& =\langle\mathcal{B}_{q}^{n+1} e_{r},e_{0}\rangle
\end{align*} 
where we used the induction hypothesis in the second equality. This concludes the proof of \eqref{reloppath2}. For $0\leq j\leq p$, we have $\mathcal{S}_{2}(n,j)=\widehat{\mathcal{D}}_{[n,j]}^{(q)}$, hence $\langle\mathcal{B}_{q}^{n} e_{j},e_{0}\rangle=B_{[n,j]}^{(q)}$ and \eqref{eq:idbetaBj} follows.

For $r\in\mathbb{Z}$ and $n\in\mathbb{Z}_{\geq 0}$ we define $\mathcal{S}_{3}(n,r)$ to be the collection of all lattice paths with initial point $(0,0)$ and terminal point $(n,r)$. We leave to the reader to check that 
\[
\langle\mathcal{W}^{n}\hat{e}_{r},\hat{e}_{0}\rangle=\sum_{\gamma\in\mathcal{S}_{3}(n,r)}w(\gamma)
\]
from which \eqref{eq:idpsiWj} can be deduced.
\end{proof}

From Proposition~\ref{prop:eqresollat} and Theorems~\ref{theo:relseriesAs} and \ref{theo:relseriesWs} we deduce:

\begin{corollary}\label{cor:relresol}
The following relations hold between the resolvent functions defined in \eqref{def:resolphij}, \eqref{def:resolbetaj}, and \eqref{def:resolpsij}. For every $k\geq 0$,
\begin{align*}
\phi_{0}^{(k)}(z) & =\frac{1}{z-a_{k}^{(0)}-\sum_{j=1}^{p}a_{k}^{(j)}\,\phi_{j-1}^{(k+1)}(z)}\\
\phi_{j}^{(k)}(z) & =\phi_{0}^{(k)}(z)\,\phi^{(k+1)}_{j-1}(z)\qquad 1\leq j\leq p.
\end{align*}
We also have
\begin{align*}
\psi_{0}(z) &=\frac{1}{z-a_{0}^{(0)}-\sum_{j=1}^{p}\sum_{k=0}^{j}a_{-k}^{(j)}\,\phi_{j-k-1}^{(1)}(z)\, \beta_{k-1}^{(1)}(z)} \\
\psi_{j}(z) & =\psi_{0}(z)\,\phi_{j-1}^{(1)}(z)\qquad 1\leq j\leq p,
\end{align*}
where in the third formula $\phi_{-1}^{(1)}(z)\equiv \beta_{-1}^{(1)}(z)\equiv 1$.
\end{corollary}

A non-combinatorial proof of the formulas in Corollary~\ref{cor:relresol} can be found in \cite{LopPro2}. 

Let $\mathbf{F}=\mathbb{C}((z^{-1}))$. In the rest of this section we discuss the expansion in vector continued fraction of the vector of resolvent functions $(\phi_{0}(z),\ldots,\phi_{p-1}(z))=(A_{0}(z),\ldots,A_{p-1}(z))\in \mathbf{F}^{p}$ by means of the Jacobi--Perron algorithm. In particular we give a proof of \eqref{vcfphifinite}.   

In addition to the division operation \eqref{divoper}, in $\mathbf{F}^{p}$ we also define the multiplication 
\[
(x_{1},\ldots,x_{p})\cdot(y_{1},\ldots,y_{p}):=(x_{1}\,y_{1},\ldots,x_{p}\,y_{p}).
\]
For simplicity of notation, the neutral element for multiplication, which is the vector with all components equal to one, will be denoted
\[
\mathbf{1}=(1,\ldots,1).
\]
Note that
\[
\frac{(x_{1},\ldots,x_{p})}{(y_{1},\ldots,y_{p})}=(x_{1},\ldots,x_{p})\cdot\frac{\mathbf{1}}{(y_{1},\ldots,y_{p})},
\]
and if $x_{1}\neq 0$, then
\begin{equation}\label{invdiv}
(x_{1},x_{2},\ldots,x_{p})=\frac{\mathbf{1}}{(\frac{x_2}{x_1},\frac{x_3}{x_1},\ldots,\frac{x_{p}}{x_{1}},\frac{1}{x_1})}.
\end{equation}
Recall the notation \eqref{fvcf}.

To facilitate the understanding of the construction of the continued fraction for the vector $(A_{0}(z),\ldots,A_{p-1}(z))$, we give first an informal discussion of the main steps of the process. To simplify notation, in what follows we will not indicate the variable $z$ unless it is necessary. 

First, according to \eqref{eq:A0Aks}--\eqref{eq:AjA0Ajm}, we have
\begin{align}
(A_{0},A_{1},\ldots,A_{p-1}) & =\left(\frac{1}{z-\sum_{j=0}^{p}a_{0}^{(j)}A_{j-1}^{(1)}},\frac{A_{0}^{(1)}}{z-\sum_{j=0}^{p}a_{0}^{(j)}A_{j-1}^{(1)}},\ldots,\frac{A_{p-2}^{(1)}}{z-\sum_{j=0}^{p}a_{0}^{(j)}A_{j-1}^{(1)}}\right)\notag\\
& =\frac{\mathbf{1}}{(A_{0}^{(1)},\ldots,A_{p-2}^{(1)},z-\sum_{j=0}^{p} a_{0}^{(j)} A_{j-1}^{(1)})}\notag\\
& =\frac{\mathbf{1}}{(0,\ldots,0,z-a_{0}^{(0)})+(A_{0}^{(1)},\ldots,A_{p-2}^{(1)},-\sum_{j=1}^{p}a_{0}^{(j)} A_{j-1}^{(1)})}\label{eqstart}
\end{align}
recall that by definition $A_{-1}^{(1)}\equiv 1$. In the last step in the above computation we have separated the polynomial part from the principal part in the Laurent series obtained in the denominator after the division. This separation step is done in each iteration of the algorithm. Now we repeat this inversion procedure for the vector $(A_{0}^{(1)},\ldots,A_{p-2}^{(1)},-\sum_{j=1}^{p}a_{0}^{(j)} A_{j-1}^{(1)})$. Using \eqref{invdiv}, we have
\begin{gather*}
(A_{0}^{(1)},\ldots,A_{p-2}^{(1)},-\sum_{j=1}^{p}a_{0}^{(j)} A_{j-1}^{(1)})=\frac{\mathbf{1}}{\left(\frac{A_{1}^{(1)}}{A_{0}^{(1)}},\ldots,\frac{A_{p-2}^{(1)}}{A_{0}^{(1)}},-\frac{\sum_{j=1}^{p} a_{0}^{(j)}\,A_{j-1}^{(1)}}{A_{0}^{(1)}},\frac{1}{A_{0}^{(1)}}\right)}\\
=\frac{\mathbf{1}}{(A_{0}^{(2)},\ldots,A_{p-3}^{(2)},-a_{0}^{(1)}-\sum_{j=2}^{p}a_{0}^{(j)} A_{j-2}^{(2)},z-a_{1}^{(0)}-\sum_{j=1}^{p} a_{1}^{(j)} A_{j-1}^{(2)})}\\
=\frac{\mathbf{1}}{(0,\ldots,0,-a_{0}^{(1)},z-a_{1}^{(0)})+(A_{0}^{(2)},\ldots,A_{p-3}^{(2)},-\sum_{j=2}^{p} a_{0}^{(j)} A_{j-2}^{(2)},-\sum_{j=1}^{p}a_{1}^{(j)} A_{j-1}^{(2)})}
\end{gather*}
where in the second equality we applied \eqref{relA0kAkpone} for $k=1$. We do one more iteration. Applying \eqref{invdiv} and \eqref{relA0kAkpone} for $k=2$, we obtain
\begin{equation*}
(A_{0}^{(2)},\ldots,A_{p-3}^{(2)},-\sum_{j=2}^{p} a_{0}^{(j)} A_{j-2}^{(2)},-\sum_{j=1}^{p}a_{1}^{(j)} A_{j-1}^{(2)})=\frac{\mathbf{1}}{\mathbf{d}_{3}+\mathbf{v}_{3}}
\end{equation*}
where
\begin{align*}
\mathbf{d}_{3} & =(0,\ldots,0,-a_{0}^{(2)},-a_{1}^{(1)},z-a_{2}^{(0)}),\\
\mathbf{v}_{3} & =(A_{0}^{(3)},\ldots,A_{p-4}^{(3)},-\sum_{j=3}^{p}a_{0}^{(j)} A_{j-3}^{(3)},-\sum_{j=2}^{p}a_{1}^{(j)} A_{j-2}^{(3)},-\sum_{j=1}^{p}a_{2}^{(j)} A_{j-1}^{(3)}).
\end{align*}
The polynomial vectors we have obtained so far in the denominators are the vectors
\begin{gather*}
\mathbf{d}_{1}=(0,\ldots,0,z-a_{0}^{(0)})\\
\mathbf{d}_{2}=(0,\ldots,0,-a_{0}^{(1)},z-a_{1}^{(0)})\\
\mathbf{d}_{3}=(0,\ldots,0,-a_{0}^{(2)},-a_{1}^{(1)},z-a_{2}^{(0)})
\end{gather*}
and we have shown that
\[
(A_{0},A_{1},\ldots,A_{p-1})=\cfrac{\mathbf{1}}{\mathbf{d}_1+\cfrac{\mathbf{1}}{\mathbf{d}_2+\cfrac{\mathbf{1}}{\mathbf{d}_3+\mathbf{v}_3}}}.
\]
Continuing in this fashion, in each iteration we get a polynomial vector of the form
\[
\mathbf{d}_{k}:=(0,\ldots,0,-a_{0}^{(k-1)},-a_{1}^{(k-2)},\ldots,-a_{k-2}^{(1)},z-a_{k-1}^{(0)})
\]
with coefficients from the $k$th row of the matrix representation of $\mathcal{H}_{0}$, see \eqref{matrixrep}. When we reach the value $k=p$, we obtain in the denominator the expression
\[
\mathbf{d}_{p}+(-a_{0}^{(p)} A_{0}^{(p)},-\sum_{j=p-1}^{p} a_{1}^{(j)} A_{j-p+1}^{(p)},-\sum_{j=p-2}^{p} a_{2}^{(j)}A_{j-p+2}^{(p)},\ldots,-\sum_{j=1}^{p}a_{p-1}^{(j)} A_{j-1}^{(p)}).
\]
The novelty of the second vector is that the first component has the coefficient $-a_{0}^{(p)}$ in front, and this feature is not present in previous iterations. At this point, we first extract the coefficient $-a_{0}^{(p)}$ from the first component of the second vector, and apply as before the inversion procedure to the vector $(A_{0}^{(p)},-\sum_{j=p-1}^{p}a_{1}^{(j)} A_{j-p+1}^{(p)},\ldots,-\sum_{j=1}^{p}a_{p-1}^{(j)} A_{j-1}^{(p)})$, i.e., we write
\begin{equation*}
(-a_{0}^{(p)} A_{0}^{(p)},-\sum_{j=p-1}^{p} a_{1}^{(j)} A_{j-p+1}^{(p)},\ldots,-\sum_{j=1}^{p}a_{p-1}^{(j)} A_{j-1}^{(p)})=\frac{(-a_{0}^{(p)},1,\ldots,1)}{\mathbf{d}_{p+1}+\mathbf{v}_{p+1}}
\end{equation*}
where
\begin{align*}
\mathbf{d}_{p+1} & =(-a_{1}^{(p-1)},\ldots,-a_{p-1}^{(1)},z-a_{p}^{(0)}),\\
\mathbf{v}_{p+1} & =(-a_{1}^{(p)}A_{0}^{(p+1)},-\sum_{j=p-1}^{p}a_{2}^{(j)} A_{j-p+1}^{(p+1)},\ldots,-\sum_{j=1}^{p} a_{p}^{(j)}A_{j-1}^{(p+1)}),
\end{align*}
and we have used \eqref{relA0kAkpone} for $k=p$. From this point on, one can repeat this last described procedure indefinitely. 

Now we proceed to the formal justification of the construction. Let $\mathbf{c}_{k}$, $\mathbf{d}_{k}(z)$, and $\mathbf{v}_{k}(z)$ be the vectors defined in \eqref{cks}--\eqref{vks}. Note that
\[
\mathbf{v}_{k}(z):=\begin{cases}
(A_{0}^{(k)},\ldots,A_{p-k-1}^{(k)},-\sum_{j=k}^{p} a_{0}^{(j)}A_{j-k}^{(k)},\ldots,-\sum_{j=1}^{p} a_{k-1}^{(j)} A_{j-1}^{(k)}), & 0\leq k\leq p,\\[0.5em]
(-a_{k-p}^{(p)} A_{0}^{(k)},-\sum_{j=p-1}^{p} a_{k-p+1}^{(j)} A_{j-p+1}^{(k)},\ldots,-\sum_{j=1}^{p} a_{k-1}^{(j)} A_{j-1}^{(k)}), & k\geq p+1,
\end{cases}
\]
and $\mathbf{v}_{0}(z)=(A_{0}(z),\ldots,A_{p-1}(z))$. 
\begin{lemma}
For each $k\geq 0$, we have
\begin{equation}\label{lft}
\mathbf{v}_{k}(z)=\frac{\mathbf{c}_{k+1}}{\mathbf{d}_{k+1}(z)+\mathbf{v}_{k+1}(z)}.
\end{equation}
\end{lemma}
\begin{proof}
If $0\leq k\leq p-1$, applying \eqref{relA0kAkpone} we get 
\begin{gather*}
\mathbf{v}_{k}(z)=(A_{0}^{(k)},\ldots,A_{p-k-1}^{(k)},-\sum_{j=k}^{p}a_{0}^{(j)} A_{j-k}^{(k)},\ldots,-\sum_{j=1}^{p}a_{k-1}^{(j)} A_{j-1}^{(k)})\\
=\frac{\mathbf{1}}{\left(\frac{A_{1}^{(k)}}{A_{0}^{(k)}},\ldots,\frac{A_{p-k-1}^{(k)}}{A_{0}^{(k)}},-\frac{\sum_{j=k}^{p}a_{0}^{(j)} A_{j-k}^{(k)}}{A_{0}^{(k)}},\ldots,-\frac{\sum_{j=1}^{p} a_{k-1}^{(j)} A_{j-1}^{(k)}}{A_{0}^{(k)}},\frac{1}{A_{0}^{(k)}}\right)}=\frac{\mathbf{c}_{k+1}}{\mathbf{d}_{k+1}(z)+\mathbf{v}_{k+1}(z)}.
\end{gather*}
If $k\geq p$, then
\begin{gather*}
\mathbf{v}_{k}(z)=(-a_{k-p}^{(p)},1,\ldots,1)\,(A_{0}^{(k)},-\sum_{j=p-1}^{p}a_{k-p+1}^{(j)} A_{j-p+1}^{(k)},\ldots,-\sum_{j=1}^{p} a_{k-1}^{(j)} A_{j-1}^{(k)})\\
=\frac{(-a_{k-p}^{(p)},1,\ldots,1)}{\left(-\frac{\sum_{j=p-1}^{p}a_{k-p+1}^{(j)}A_{j-p+1}^{(k)}}{A_{0}^{(k)}},\ldots,-\frac{\sum_{j=1}^{p}a_{k-1}^{(j)}A_{j-1}^{(k)}}{A_{0}^{(k)}},\frac{1}{A_{0}^{(k)}}\right)}=\frac{\mathbf{c}_{k+1}}{\mathbf{d}_{k+1}(z)+\mathbf{v}_{k+1}(z)}.
\end{gather*}
\end{proof}
It follows from \eqref{lft} that for every $n\geq 1$,
\[
(A_{0}(z),\ldots,A_{p-1}(z))=\Kont_{m=1}^{n}\Bigl(\frac{\mathbf{c}_{m}}{\widetilde{\mathbf{d}}_{m}(z)}\Bigr)
\]
where $\widetilde{\mathbf{d}}_{m}(z)=\mathbf{d}_{m}(z)$ if $m\leq n-1$ and $\widetilde{\mathbf{d}}_{n}(z)=\mathbf{d}_{n}(z)+\mathbf{v}_{n}(z)$. It is in this sense that we write the formal identity of Kalyagin: 
\[
(A_{0}(z),\ldots,A_{p-1}(z))=\Kont_{n=1}^{\infty}\Bigl(\frac{\mathbf{c}_{n}}{\mathbf{d}_{n}(z)}\Bigr).
\]

In the particular case of the functions $S_{j}(z)$ defined in \eqref{fsSj} (see also \eqref{SRedges} and \eqref{eq:ananp}) associated with the generalized Stieltjes--Rogers polynomials, the vector continued fraction takes the form
\[
(S_{0}(z),\ldots,S_{p-1}(z))=\Kont_{n=1}^{\infty}\Bigl(\frac{\widehat{\mathbf{c}}_{n}}{\widehat{\mathbf{d}}_{n}(z)}\Bigr)
\]
where
\begin{align*}
\widehat{\mathbf{c}}_{n} & =\begin{cases}
\mathbf{1}, & 1\leq n\leq p,\\
(-a_{n-p-1},1,\ldots,1), & n\geq p+1,
\end{cases}\\
\widehat{\mathbf{d}}_{n}(z) & =(0,\ldots,0,z), \quad n\geq 1.
\end{align*}
It was shown by Aptekarev-Kaliaguine-Van Iseghem in \cite[Thm. 3]{AptKalVan} that an alternative formal expansion can be given for the vector $(S_{0}(z),\ldots,S_{p-1}(z))$. It follows from
\[
(S_{0}(z),\ldots,S_{p-1}(z))=\frac{\mathbf{1}}{(0,\ldots,0,z)+(1,\ldots,1,-a_{0})\,(S_{0}^{(1)}(z),\ldots,S_{p-1}^{(1)}(z))},
\]   
see \eqref{eqstart}. Since a similar relation holds between the functions $S_{j}^{(k)}$ and $S_{j}^{(k+1)}$ for any $k$, we get
\[
(S_{0}(z),\ldots,S_{p-1}(z))=\cfrac{\mathbf{1}}{(0,\ldots,0,z)+\cfrac{(1,\ldots,1,-a_{0})}{(0,\ldots,0,z)+\cfrac{(1,\ldots,1,-a_{1})}{\ddots}}}.
\]

\bigskip

\noindent\textbf{Acknowledgements:} We thank the referee for his comments to improve the presentation of this paper.

\bigskip

\noindent \textsc{Department of Mathematics, University of Central Florida, 4393 Andromeda Loop North, Orlando, FL 32816, USA} \\
\textit{Email address}: \texttt{abey.lopez-garcia\symbol{'100}ucf.edu}

\bigskip

\noindent \textsc{Department of Mathematics and Statistics, University of South Alabama, 411 University Boulevard North, 
Mobile, AL 36688, USA}\\
\textit{Email address}: \texttt{prokhoro\symbol{'100}southalabama.edu}

\end{document}